\documentclass{article}

\PassOptionsToPackage{round}{natbib}
\usepackage[preprint]{neurips_2026}

\usepackage[utf8]{inputenc} % allow utf-8 input
\usepackage[T1]{fontenc}    % use 8-bit T1 fonts
\usepackage{hyperref}       % hyperlinks
\usepackage{url}            % simple URL typesetting
\usepackage{booktabs}       % professional-quality tables
\usepackage{amsfonts}       % blackboard math symbols
\usepackage{nicefrac}       % compact symbols for 1/2, etc.
\usepackage{microtype}      % microtypography
\usepackage{xcolor}         % colors

\title{Curvature-Dependent Lower Bounds for Frank--Wolfe}

\author{%
  Jannis Halbey\\
  Zuse Institute Berlin \& TU Berlin\\
  Berlin, Germany \\
  \texttt{halbey@zib.de}
  \And
  Christophe Roux\\
  Zuse Institute Berlin \& TU Berlin\\
  Berlin, Germany \\
  \texttt{roux@zib.de}
  \And
  Sebastian Pokutta\\
  Zuse Institute Berlin \& TU Berlin\\
  Berlin, Germany \\
  \texttt{pokutta@zib.de}
}

\usepackage{amsmath,amssymb,amsthm,mathtools}
\usepackage{aliascnt}
\usepackage{booktabs}
\usepackage{graphicx}
\usepackage{subcaption}
\usepackage{xcolor}
\definecolor{darkblue}{RGB}{0,0,139}
\usepackage{hyperref}
\hypersetup{
  colorlinks=true,
  linkcolor=darkblue,
  citecolor=darkblue,
  urlcolor=darkblue
}
\usepackage[most]{tcolorbox}

\usepackage{thmtools, thm-restate}
\usepackage{algorithm}
\usepackage{algpseudocode}
\declaretheorem{theorem}
\declaretheorem[sibling=theorem]{lemma}

\declaretheorem[sibling=theorem]{proposition}
\declaretheorem[sibling=theorem]{corollary}
\declaretheorem[sibling=theorem]{definition}

\usepackage[capitalize,noabbrev,nameinlink]{cleveref}
\newcommand{\R}{\mathbb{R}}
\newcommand{\C}{\mathcal{C}}

\DeclareMathOperator{\spann}{span}
\DeclareMathOperator{\sign}{sign}
\DeclareMathOperator*{\argmin}{argmin}
\DeclareMathOperator{\supp}{supp}
\newcommand{\norm}[1]{\|#1\|}
\newcommand{\ip}[1]{\langle #1 \rangle}

\newcommand{\bigo}[1]{\mathcal{O}( #1 )}

\newcommand{\bigomega}[1]{\Omega( #1 )}

\newcommand\linktoproof[1]{{\hspace{-0.3em}\normalfont[{\hyperlink{proof:#1}{$\downarrow$}}]}}
\newcommand{\defi}{\stackrel{\mathrm{\scriptscriptstyle def}}{=}}

\newcommand{\newtarget}[1]{\hypertarget{#1}{}}
\newcommand\linkofproof[1]{(\cref{#1})\newtarget{proof:#1}\ }

\usepackage{todonotes}

\usepackage[acronym]{glossaries}
\newacronym{fw}{FW}{Frank--Wolfe}
\newacronym{lmo}{LMO}{Linear Minimization Oracle}
\newacronym{heb}{HEB}{H\"olderian Error Bound}

\begin{document}

\maketitle

\begin{abstract}
% The \gls{fw} algorithm achieves a convergence rate of $\bigo{1/T}$ on smooth convex objectives over compact convex sets and accelerates to $\bigo{1/T^2}$ when both the feasible set and the objective are strongly convex.
% \citet{kerdreux2021projection} established improved convergence rates of $\bigo{T^{-p/(p-1)}}$ when the feasible set is $p$-uniformly convex, which subsume strongly convex sets and ($p=2$) and exemplified by $\ell_p$ balls.
% Uniformly convex sets non-trivially subsume strongly convex sets and form a large variety of curved convex sets . For instance, the `p-balls are uniformly convex for all p > 1, but strongly convex for p ∈]1, 2] only. We show that these sets systematically induce accelerated convergence rates for the original \gls{fw} algorithm, which continuously interpolate between known rates.
% We show that these rates are tight for every $p\ge 3$ when using exact line search or short steps.
% Further, we extend the lower bound from strongly convex functions to function satisfying a Hölderian error bounds.
% Note that the lower bounds are not limited to the high-dimensional setting, as they are based on analyzing the dynamics of the \gls{fw} iterates on a simple function instead of information-theoretic arguments.
The \gls{fw} algorithm achieves a convergence rate of $\bigo{1/T}$ for smooth convex optimization over compact convex domains, accelerating to $\bigo{1/T^2}$ when both the objective and the feasible set are strongly convex. This acceleration extends beyond strong convexity: \citet{kerdreux2021projection} proved rates of $\bigo{T^{-p/(p-1)}}$ over $p$-uniformly convex feasible sets, a class that interpolates between strongly convex sets and more general curved domains such as $\ell_p$ balls. In this work, we establish a matching $\bigomega{T^{-p/(p-1)}}$ lower bound for every $p\ge 3$ under exact line search or short steps, and extend the lower bound to objectives satisfying a Hölderian error bound. The proofs analyze the dynamics of \gls{fw} iterates on simple instances and hence are not limited to the high-dimensional setting, unlike information-theoretic lower bounds.
\end{abstract}

\section{Introduction}
The \acrfull{fw} algorithm \citep{frank1956algorithm} is designed to solve constrained optimization problems of the form
\begin{equation*}
  \min_{x\in\C} f(x),
\end{equation*}
using only access to a \gls{lmo}. Here $f:\R^d\to \R$ is a smooth and convex function and $\C\subseteq \R^d$ is a compact and convex set.

The classical $\bigo{1/T}$ rate \citep{jaggi13revisiting} is tight for general compact convex $\C $ \citep{canon1968tight,lan2013complexity}, but stronger structural assumptions yield faster rates.
\citet{levitin1966constrained} and \citet{wolfe1970convergence} obtained linear rates under conditions on the position of the unconstrained optimizer, and \citet{garber15faster} obtained the global rate $\bigo{1/T^2}$ when both the set and the objective are strongly convex.
\citet{halbey2026lower} recently showed that it is tight for \gls{fw} with exact line search or short steps, and \citet{grimmer2026uniform} obtained a matching $\bigomega{1/T^2}$ bound for a broader class of \gls{lmo}-based methods in the high-dimensional regime.

\citet{kerdreux2021projection} extended the upper bound picture from strongly convex sets to uniformly convex sets, and from strongly convex objectives to objectives satisfying a \gls{heb}, deriving a family of rates that interpolate between $\bigo{1/T}$ and $\bigo{1/T^2}$.
Whether these interpolating rates are tight has remained open.
In this work we resolve the question partially by establishing a matching lower bound for \gls{heb} objectives with $\theta=\tfrac{1}{2}$ on uniformly convex sets, and a looser lower bound for general $\theta \in (0, \tfrac{1}{2})$.

Both lower bounds follow from the same recipe. We first minimize a quadratic objective over the $\ell_p$ ball and exhibit an explicit \emph{slow-start} initialization such that exact-line-search \gls{fw} iterates remain pinned to a particular slow trajectory.
Tracking the contraction along this trajectory recovers the predicted rate.
We then replace this quadratic objective with a richer family of objectives and apply the same argument to obtain the lower bound for the $(\mu,\theta)$-\gls{heb} regime.

Our lower bound results are closely related to the one of \citet{halbey2026lower}.
We use a similar model setup, replacing the Euclidean ball with the unit $\ell_p$ ball, and obtain the lower bound by analyzing specific \gls{fw} trajectories.
However, the method for initializing the starting point and the proof technique differ significantly from that used by \citet{halbey2026lower}.
We give a detailed comparison at the end of \Cref{sec:lower-bound-uc}.

The proof of the main result in this paper (\cref{thm:lower-bound}) was found using the agentic research framework that we introduced in a companion preprint \citep{zimmer2026agentic}, which delegates well-scoped research subtasks to LLM agents under structured human supervision.
Section~4.4 of \citet{zimmer2026agentic} states \Cref{thm:lower-bound} without proof, as one of the case studies that motivated the framework.
The present paper closes that gap by providing the precise theorem statement, the asymptotic constants, the full proof, the \gls{heb} extension (\Cref{thm:extended-lower-bound}), and the numerical validation. \Cref{sec:discussion} discusses which steps were agent-led, which were load-bearing for human judgement, and the failure modes we encountered.

\paragraph{Contributions.}
\begin{itemize}
  \item For every $p\ge 3$, we prove a lower bound of $\bigomega{T^{-p/(p-1)}}$ for \gls{fw} with exact line search or short steps on smooth strongly convex objectives over $p$-uniformly convex sets, matching the upper bound of \citet{kerdreux2021projection,pokutta2026frankwolfe} and resolving the tightness question in this regime (\Cref{thm:lower-bound}).
  \item For $(\mu,\theta)$-\gls{heb} objectives on $p$-uniformly convex sets with $p\geq 3$ and $\theta \in (0, 1/2]$, we establish a lower bound of $\bigomega{T^{-p/(2\theta(p-1))}}$ (\Cref{thm:extended-lower-bound}).
  \item We document our methodology for developing the proofs with agentic AI assistance, contributing a case study to the emerging practice of AI-assisted proof discovery (\Cref{sec:discussion}).
\end{itemize}

\paragraph{Related work.}
\Cref{tab:rates} positions our results among the rates known for vanilla \gls{fw} with short steps or an exact line search. We discuss the four axes that govern the picture (curvature of $\C $, regularity of $f$, position of the optimizer, and step-size rule) in turn.

\emph{Upper bounds.}
The earliest accelerations rely on the position of the optimizer rather than uniform set geometry.
\citet{wolfe1970convergence} proved linear convergence whenever $f$ is strongly convex and $x^*$ lies in the interior of $\C $, with no curvature assumption on $\C $.
\citet{levitin1966constrained} proved linear convergence on a strongly convex $\C $ whenever $\norm{\nabla f}_\ast$ is bounded below on $\C $ (equivalently, the unconstrained optimum lies outside $\C $).
\citet{garber15faster} removed the position assumption and obtained the uniform $\bigo{1/T^2}$ rate when $\C $ and $f$ are both strongly convex.
\citet{kerdreux21affine} recovered this rate from an affine-invariant analysis, and \citet{wirth23acceleration} extended it to open-loop step sizes.
\citet{kerdreux2021projection} proved $\bigo{T^{-p/(p-1)}}$ rates on $p$-uniformly convex sets and refined them under \gls{heb} on $f$.
\citet{pokutta2026frankwolfe} introduced \emph{local dual sharpness}, an oracle-geometry condition implied by uniform convexity, and proved the first unconditional $o(1/T)$ rate for vanilla \gls{fw} on such sets with matching quantitative tails under a local \gls{heb}.

\emph{Lower bounds.}
The classical $\bigomega{1/T}$ bounds \citep{canon1968tight,lan2013complexity} are built on polytopal or general convex constraint sets and do not apply to curved sets.
Until recently, it was open whether the rates obtained by \citet{garber15faster} and \citet{kerdreux2021projection} are tight.
\citet{halbey2026lower} showed a matching lower bound for strongly convex functions and sets by analyzing the trajectory of the \gls{fw} iterates on a quadratic on the Euclidean ball.
Concurrently, \citet{grimmer2026uniform} proved a matching $\bigomega{1/T^2}$ for the broader class of \gls{lmo}-span methods via a high-dimensional zero-chain construction.
The two results are complementary: \citet{grimmer2026uniform} covers a wider algorithm class but is restricted to the regime $T<d$, and its extension to smooth sets applies only to \emph{modestly} smooth sets whose smoothness constant scales with the target precision as $L=\Theta(1/\sqrt{\varepsilon})$.
While \citet{halbey2026lower} targets a narrower class of methods, it provides a dimension-independent result. 

\emph{Adjacent settings.}
Linear convergence of \gls{fw} variants on polytopes (away-step, pairwise, and fully-corrective methods, see \citet{braun2025conditional} for an overview) relies on polyhedral geometry rather than set curvature and is orthogonal to our setting.

\begin{table}[t]
\centering
\caption{Upper and lower bounds on convergence rates for \gls{fw} with short steps or exact line search. All entries assume that the constraint set $\C$ is convex and compact and that $f$ is convex and smooth. The abbreviation SC refers to strongly convex and UC to uniformly convex. The entries in column $\C$ and $f$ denote the additional assumptions on the constraint set and the function, respectively.\\}
\label{tab:rates}
\begin{tabular}{lccc}
\toprule
Work & $\C $ & $f$ & Result \\
\midrule
\citet{jaggi13revisiting}                            & --              &   --              & $\bigo{T^{-1}}$ \\
\citet{lan2013complexity}             & --              & --                & $\bigomega{T^{-1}}$ \\
\midrule
\citet{kerdreux2021projection} & $p$-UC           & SC                    & $\bigo{T^{-\frac{p}{p-1}}}$ \\
\Cref{thm:lower-bound} (ours)                        & $p$-UC, $p\ge 3$ & SC                    & $\bigomega{T^{-\frac{p}{p-1}}}$ \\
\citet{kerdreux2021projection}                       & $p$-UC           & $\theta$-\gls{heb}    & $\bigo{T^{-\frac{p}{p-2\theta}}}$ \\
\Cref{thm:extended-lower-bound} (ours)                  & $p$-UC, $p\ge 3$ & $\theta$-\gls{heb}  & $\bigomega{T^{-\frac{p}{2\theta(p-1)}}}$ \\
\midrule
\citet{garber15faster}                               & SC               & SC                    & $\bigo{T^{-2}}$ \\
\citet{halbey2026lower,grimmer2026uniform}           & SC               & SC                    & $\bigomega{T^{-2}}$ \\
\bottomrule
\end{tabular}
\end{table}

\paragraph{Preliminaries.}
We work in a finite-dimensional normed space $(\R^d,\norm{\cdot})$. 
The support of a vector $x \in \R^d$ is defined as $\supp(x) \defi  \{ i \in [d]: x_i \neq 0 \}$.
A differentiable function is \emph{$\mu$-strongly convex} if
\[
    f(x) \geq f(y) + \ip{\nabla f(y), x-y} + \frac{\mu}{2}\norm{x-y}^2
\]
for all $x,y \in \R^d$ and \emph{$L$-smooth} if
\[
    f(x) \leq f(y) + \ip{\nabla f(y), x-y} + \frac{L}{2}\norm{x-y}^2.
\]
For a smooth convex objective $f:\C \to\R$ on a compact convex set $\C \subset\R^d$, we denote an optimizer by
$x^*\in\arg\min_{x\in \C }f(x)$ and the primal suboptimality by $h_t\defi f(x_t)-f(x^*)$.
Further, $B_p$ denotes the unit $\ell_p$ ball, i.e., $ B_p = \{x \in \R^d : \norm{x}_p \leq 1\}$.

\begin{definition}[Uniformly convex set {\citep{kerdreux2021projection}}]
    Let $\C \subset\R^d$ be closed and $\gamma_\C :\R_+\to\R_+$ be nondecreasing.
    We say that $\C $ is $\gamma_\C $-uniformly convex (with respect to $\norm{\cdot}$) if
    for all $x,y\in \C $, all $\eta\in[0,1]$, and all $z\in\R^d$ with $\norm{z}=1$,
    \[
    \eta x + (1-\eta)y + \eta(1-\eta)\gamma_\C (\norm{x-y})\,z \in \C .
    \]
    If there exist $\alpha>0$ and $q\ge 2$ such that
    $\gamma_\C (r)\ge \alpha r^q$ for all $r\ge 0$, then $\C $ is called
    $(\alpha,q)$-uniformly convex (or $q$-uniformly convex).
\end{definition}

\begin{definition}[H\"olderian error bound {\citep{kerdreux2021projection}}]
    Let $f$ be strictly convex on $\C $, and let $x^*=\arg\min_{x\in \C }f(x)$.
    We say that $f$ is $(\mu,\theta)$-\gls{heb} (or $\theta$-\gls{heb} for short) if there exist
    $\mu>0$ and $\theta\in[0,\tfrac12]$ such that
    \[
    \norm{x-x^*}\le \mu\bigl(f(x)-f(x^*)\bigr)^\theta,\qquad \forall x\in \C .
    \]
\end{definition}
%% ====================================================================
\section{Lower bound for uniformly convex sets}\label{sec:lower-bound-uc}
%% ====================================================================

In this section, we establish lower bounds on the convergence rate of \gls{fw} on $p$-uniformly convex constraint sets with $p \geq 3$.
We consider the problem of projecting the point $e_1 = (1,0,\dots,0)$ onto the unit $\ell_p$ ball for $p\geq 3$,
\begin{equation}\label{eq:model}
    \min_{x\in B_p}\; f(x) \defi  \norm{x-e_1}^2.
\end{equation}
The quadratic objective is smooth and strongly convex with matching constants $L=\mu=2$.
Given the unique minimizer $x^* = e_1$, the function $f$ also satisfies the \gls{heb} for $\theta = \frac12$.
Furthermore, the constraint set is $p$-uniformly convex.
Despite its simplicity, this instance turns out to be fundamentally challenging for the \gls{fw} algorithm.
Before deriving the lower bound, we briefly recall the \gls{fw} method in \cref{alg:fw}, with the two step-size rules considered in this work: an exact line search and the short step rule.
\begin{algorithm}[H]
    \caption{\gls{fw} with exact line search or short step \label{alg:fw}}
    \begin{algorithmic}[1]
        \Require initial point $x_0\in\C $; smoothness constant $L$ of $f$ (short step only).
        \For{$t=0,1,2,\dots$}
        \State $v_t \in \arg\min_{v\in \C }\ip{\nabla f(x_t),v}$
        \State Choose step size $\gamma_t\in[0,1]$ by either rule:
        \Statex \qquad\quad $\gamma_t\in\arg\min_{\gamma\in[0,1]} f\bigl(x_t+\gamma(v_t-x_t)\bigr)$ \Comment{exact line search}
        \Statex \qquad\quad  $\gamma_t = \min\!\left\{1,\ \dfrac{\ip{\nabla f(x_t),x_t-v_t}}{L\,\norm{x_t-v_t}_2^2}\right\}$ \Comment{short step}
        \State $x_{t+1}=x_t+\gamma_t(v_t-x_t)$
        \EndFor
    \end{algorithmic}
\end{algorithm}

As a first step we consider the \gls{lmo} on the $\ell_p$ ball.
% By H\"older duality, the oracle solution is obtained by normalizing the signed coordinate-wise powers of the gradient.
Let $q=p/(p-1)$ be the H\"older conjugate, then for any nonzero direction $g\in\R^d$, the \gls{lmo} on $B_p$ produces the unique minimizer $v$ with the closed form
\begin{equation}
    v_i = -\frac{\operatorname{sign}(g_i)\,|g_i|^{q-1}}{\norm{g}_q^{q-1}},
    \qquad i=1,\dots,d.
    \label{eq:lmo-lp}
\end{equation}
The proof is given in \Cref{lem:lmo-lp} in \Cref{sec:proofs}.
\Cref{eq:lmo-lp} immediately reveals that the \gls{lmo} preserves the support, i.e., $\supp(v) = \supp(g)$.
Together with the fact that \gls{fw} iterates are always convex combinations of the extreme points of the oracle and the previous iterates, we get the following confinement result.
% For problem \eqref{eq:model}, $\nabla f(x_t)=2(x_t-e_1)$, so if $x_t\in\spann\{e_1,e_2\}$ then $\nabla f(x_t)$ is supported only on coordinates $1$ and $2$, hence the same holds for $v_t$.
% Since the \gls{fw} update in \Cref{alg:fw} is a convex combination of $x_t$ and $v_t$, every iterate stays in $\spann\{e_1,e_2\}$.
\begin{restatable}{proposition}{confinementProposition}\label{prop:confinement}\linktoproof{prop:confinement}
    Let $(x_t)_{t\ge 0}$ be the sequence generated by \gls{fw} in \Cref{alg:fw} for problem \eqref{eq:model}.
    If $x_0\in\spann\{e_1,e_2\}$, then $x_t\in\spann\{e_1,e_2\}$ for every $t\ge 0$.
\end{restatable}

Thus the problem dynamics can be described in a two-dimensional subspace whenever we initialize in $\spann\{e_1,e_2\}$.
Therefore, we focus on $d=2$ without loss of generality.

\subsection{Construction of the slow initialization}

To analyze the local behavior near the optimizer $e_1$, it is convenient to separate the error along the first coordinate from the transverse component along $e_2$.
We therefore define centered coordinates
\[
    u\defi 1-x_1,\qquad w\defi x_2.
\]
Combining the definition of the \gls{lmo} in \eqref{eq:lmo-lp} with exact line search on the quadratic objective yields an explicit one-step map $(u,w)\mapsto(u',w')$.
We characterize the resulting update in closed form in \Cref{prop:uw-dynamics} in \Cref{sec:proofs}.

The trajectories in \Cref{fig:uw-random-trajectories}, obtained from four random initializations in the $(u,w)$-plane, show that the distance to the optimum is dominated by the first coordinate $u$.
In the same runs, the transverse variable $w$ changes sign across iterations and decays faster than $u$, so the iterates quickly oscillate around the $u$-axis while drifting slowly toward the origin.

\begin{figure}[ht]
    \centering
    \includegraphics[width=0.5\linewidth]{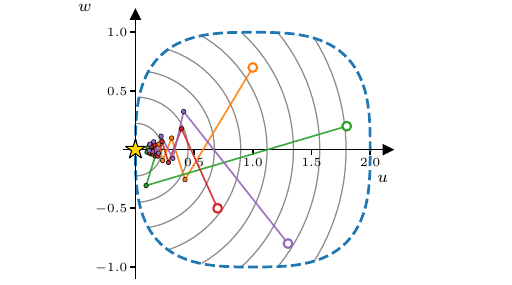}
    \caption{Exact-line-search \gls{fw} trajectories in the $(u,w)$-plane for four random initializations for the projection problem \eqref{eq:model}, i.e., $\min_{x\in B_p}\norm{x-e_1}^2$ with $p\ge 3$.
    Across runs, $u$ remains positive and controls the distance to the optimum, while $w$ alternates sign and contracts faster.}
    \label{fig:uw-random-trajectories}
\end{figure}
Therefore, we aim to derive a sharp asymptotic lower bound by expanding the one-step dynamics in powers of the dominant coordinate $u$.
The remaining question is how to scale the transverse coordinate $w$ relative to $u$.
Different approach directions $(u,w)\to(0,0)$ produce different leading orders of the \gls{fw} quantities like the step size.
So an expansion in $u$ is not uniform unless we fix a scaling regime for $w$.
To disregard sign changes and capture the faster convergence of the transverse coordinate $w$, we introduce the scaled variable
\begin{equation}
    y \defi \frac{|w|}{u^{1+\alpha}},
    \label{eq:scaled-vars}
\end{equation}
where $\alpha=(p-1)/p = 1/q$ is the reciprocal of the H\"older conjugate exponent.
We will later see that this choice of the exponent $1 + \alpha$ is critical for the lower-bound construction.
Next we consider the one-step update of the scaled variable $y$,
\begin{equation}
  \label{eq:y-one-step}
    y' = \Phi(u,y) \defi \frac{|w'|}{(u')^{1+\alpha}},
\end{equation}
where $u'$ and $w'$ are the new values of $u$ and $w$ after one \gls{fw} update.
% In \cref{lem:y-map} in \Cref{sec:proofs} we show that the one-step map $\Phi(u,y)$ admits the following expansion as $u \downarrow 0$
% \[
%     \Phi(u,y) = F(y) + u G(y) + O(u^\kappa),
% \]
% where $F(y) = y^{-(q-1)}-\frac{y}{p}$, $G(y) = \frac{y}{p}+\frac{p-1}{2p^2}\,y^{q+1}$ and $\kappa=2\alpha=2/q$. 
% Therefore the scaled variable $y$ has a constant leading-order term so $y = O(1)$ as $u \downarrow 0$.
% Note for different exponents in \eqref{eq:scaled-vars} the leading-order term of $y'$ would depend on $u$.
% In the next step, we consider the fixed point of the leading-order map $F$ to prove that the scaled variable $y$ converges to a constant as $u \downarrow 0$.
% On the interval $I$, the function $F$ has a unique fixed point at
% \[
%     y = C_p = \left(\frac{p}{p+1}\right)^{1/q}.
% \]
% Since $\Phi(0,y) = F(y)$, this fixed point is also a fixed point of the one-step map $\Phi(u,y)$ at $u=0$.
% We now extend this single fixed point to a continuous fixed point map $y^*(u)$ by the implicit function theorem.
Next we derive a fixed point map $y^*(u)$ for the function $\Phi(u,y)$ using the implicit function theorem.
There exist $u_0>0$ and a map $y^*:[0,u_0]\to(0,\infty)$ that satisfies $\Phi(u,y^*(u)) = y^*(u)$.
We show in \cref{prop:z-drift} in \Cref{sec:proofs} that the map $y^*(u)$ is Lipschitz continuous on $[0,u_0]$ and $C^1$ on $(0,u_0]$ and expands as 
\[
    y^*(u) = C_p + D_p u + \bigo{u^\kappa},
\]
for $u \downarrow 0$. The exact definitions of the constants $C_p$ and $D_p$ are given in \cref{prop:z-drift}.
This first-order description directly motivates an explicit initialization.
For small $u_0>0$, we set
\begin{equation}
    \label{eq:slow-start-init}
    x_0^{\mathrm{slow}}(u_0)\defi(1-u_0)e_1+y_0\,u_0^{1+\alpha}e_2,
    \qquad
    y_0 \defi C_p+D_pu_0.
\end{equation}
That is, we initialize on the first-order approximation of the fixed-point curve $y^*(u)$.
The chosen starting point is a feasible interior point of $B_p$ for small $u_0$ (see \Cref{lem:slow-start-feasible} in \Cref{sec:proofs}).
\Cref{fig:coordinate-systems} illustrates the resulting trajectory in the original coordinates $(x_1,x_2)$, the recentered coordinates $(u,w)$, and in the scaled variables $(u,y)$.
The last one reveals stable contraction of $y_t$ towards the curve $y^*(u)$.
\begin{figure}[H]
    \centering
    \includegraphics{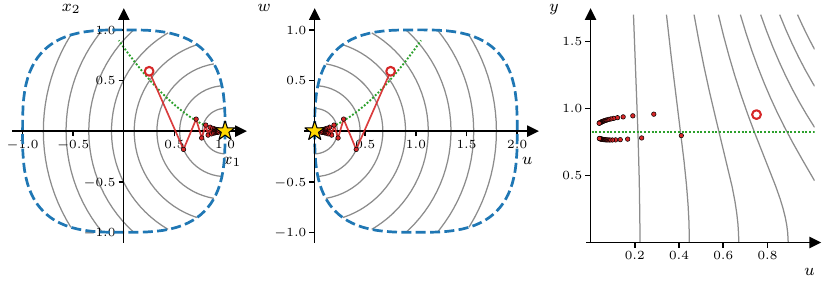}
    \caption{Three coordinate views of \gls{fw} with exact line search on $B_p$, $p=3$, for $f(x)=\norm{x-e_1}_2^2$, run from the slow-start initialization \eqref{eq:slow-start-init} with $u_0=\frac{3}{4}$ (open circle).
    Left: original coordinates $(x_1,x_2)$, with the optimum $e_1$ (gold star) and $\partial B_p$ (dashed).
    Middle: the recentered coordinates $(u,w)=(1-x_1,x_2)$ shift $e_1$ to the origin and turn $f$ into a quadratic centered at $0$, but the iterates still cluster in a tiny corner.
    Right: the scaled coordinates $(u,y)$ from \eqref{eq:scaled-vars} alternate around the slow curve $y\approx C_p$ (green dotted) which appears as a horizontal line.
    The scaling exposes the slow-curve structure that ultimately drives the $T^{-p/(p-1)}$ asymptotic rate.}
    \label{fig:coordinate-systems}
\end{figure}

In the next step we prove exactly this behavior, namely that the one-step map $\Phi(u,y)$ is contractive in the $y$-direction near $y^*(u)$.
\Cref{lem:phi-local-contraction} proves that for $p > 3$, there exist $u_0>0$ and $\lambda \in (0,1)$ such that
\[
    \left|\frac{\partial}{\partial y}\Phi(u,y)\right|\le \lambda
\]
if $y$ is close to $y^*(u)$, specifically $|y-y^*(u)|\le \delta$ for some $\delta>0$.
For $p=3$, one cannot bound the derivative by a uniform constant.
However, one can show that the derivative is bounded by $1-cu$ for some $c>0$ if $|y-y^*(u)|\le K_0 u^\kappa$ for some $K_0>0$:
\[
    \left|\frac{\partial}{\partial y}\Phi(u,y)\right|\le 1-cu.
\]
Hence, once initialized near the first-order profile of $y^*(u)$, the iterates are pulled back toward the slow curve rather than drifting away.
We formalize this quantitatively in \Cref{prop:slow-start-tracking}, showing
\begin{equation}
    \label{eq:tracking-estimate}
    |y_t-y^*(u_t)|=\bigo{u_t^\kappa}
    \qquad\text{for all }t\ge 0.
\end{equation}

Before stating the lower bound, we briefly give some remarks on the above results.
For the one-step $y$-map $\Phi(u,y)$ the case $p=3$ presents the boundary case where the contraction is not uniform.
At $p=3$, the contraction is only marginal but still quantifiable.
Numerical experiments show that for $2<p<3$, the contraction weakens further and the neighborhood of the $y^*(u)$ curve becomes more delicate.
We illustrate this phenomenon in \Cref{fig:heatmaps}, which compares heatmaps depicting the necessary iteration count to reach a fixed target accuracy of $10^{-4}$ for different values of $p$.

\begin{figure}[htbp]
    \centering
    \begin{subfigure}{0.24\linewidth}
        \includegraphics[width=\linewidth]{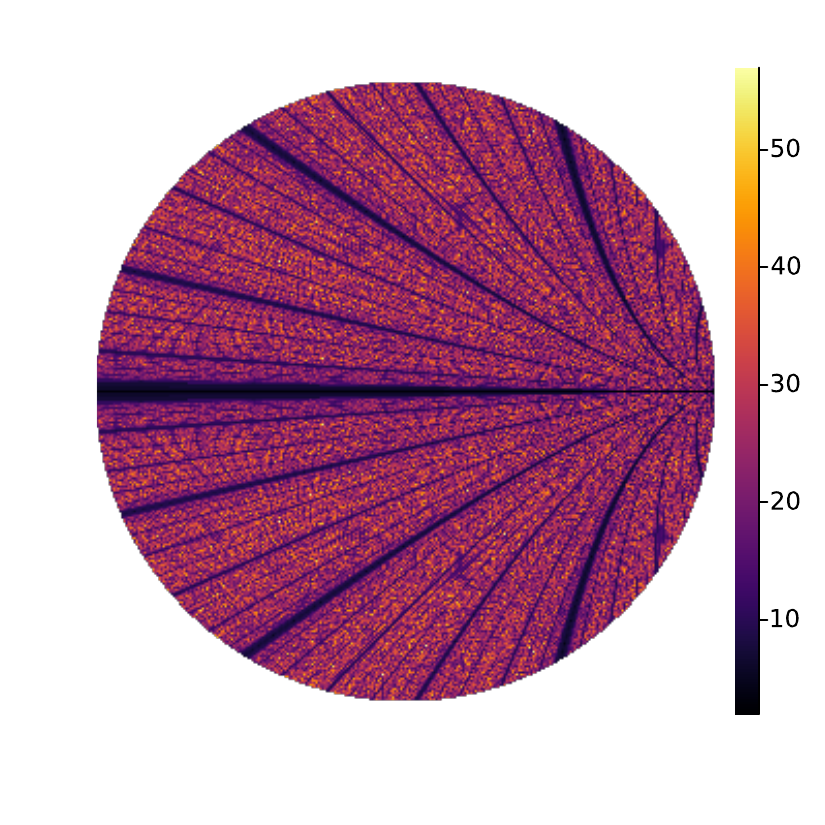}
        \caption{$p=2$}
        \label{fig:heatmap-2}
    \end{subfigure}
    \hfill
    \begin{subfigure}{0.24\linewidth}
        \includegraphics[width=\linewidth]{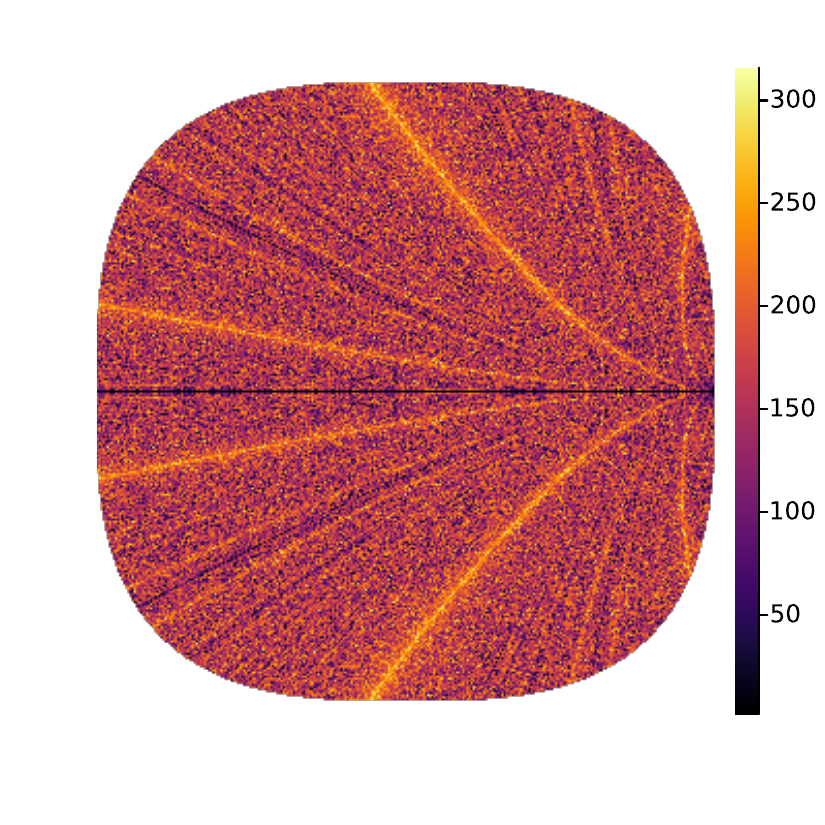}
        \caption{$p=2.95$}
        \label{fig:heatmap-2.95}
    \end{subfigure}
    \hfill
    \begin{subfigure}{0.24\linewidth}
        \includegraphics[width=\linewidth]{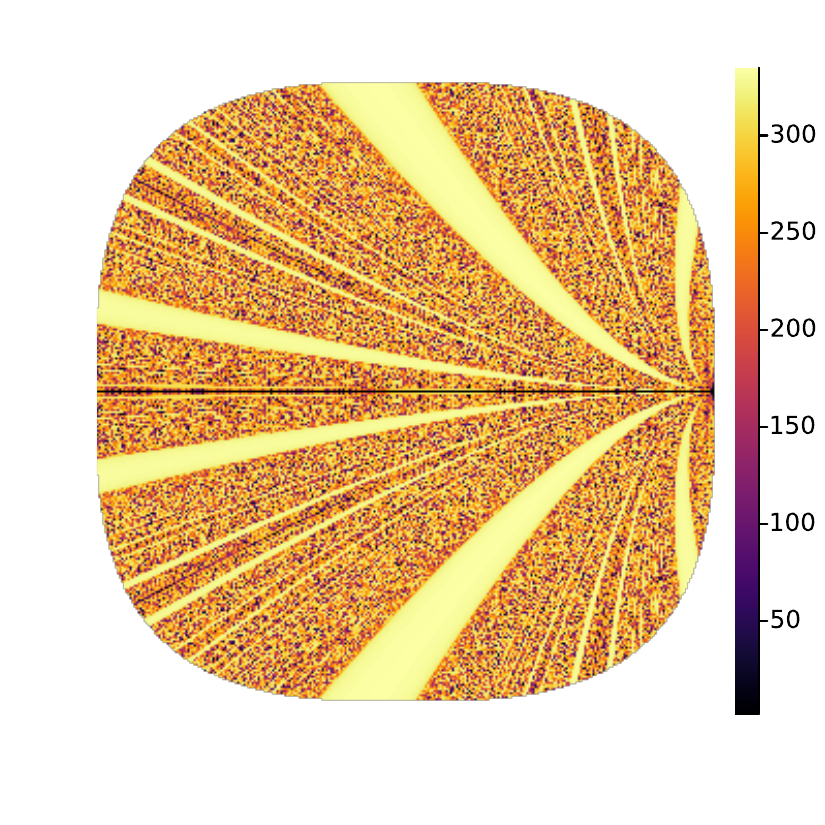}
        \caption{$p=3$}
        \label{fig:heatmap-3}
    \end{subfigure}
    \hfill
    \begin{subfigure}{0.24\linewidth}
        \includegraphics[width=\linewidth]{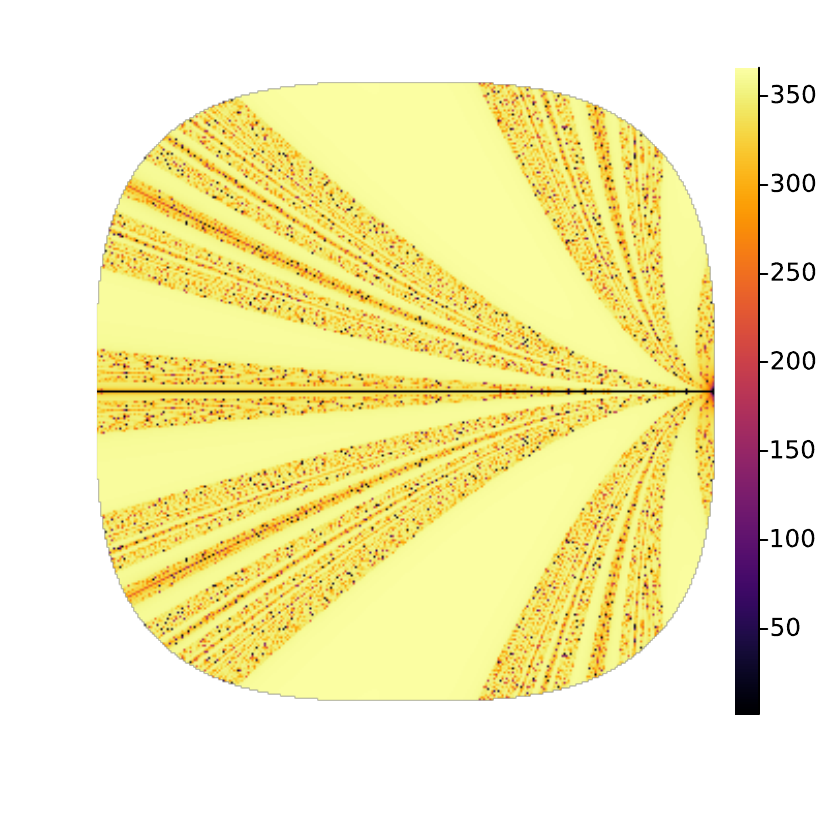}
        \caption{$p=3.1$}
        \label{fig:heatmap-3.1}
    \end{subfigure}
    \caption{
        Heatmaps showing the number of \gls{fw} iterations required to reach the target accuracy ($10^{-4}$),
        initialized from all feasible points in the open $p$-unit ball in $\mathbb{R}^2$.
        Each panel corresponds to a different value of $p$.
        Darker colors correspond to fewer iterations.
    }
    \label{fig:heatmaps}
\end{figure}

The heatmap for $p=3$ shows clear strips of high iteration counts, where the widest strips contain the fixed point curve $y^*(u)$.
The case of $p=3.1$ is exemplary for the behavior of $p>3$, showing that the strips become wider for larger $p$.
For $p<3$, here depicted by $p=2.95$, the strips become narrower making the region of high iteration counts harder to identify.
For the case of $p=2$, one cannot identify a clear strip of slow iterations, which is consistent with the results in \citet{halbey2026lower}, who studied the strongly convex case $p=2$.

\subsection{Establishing the lower bound}

We now derive the asymptotic convergence rate of the primal gap $h_t=f(x_t)-f(x^*)$ along the explicit slow-start trajectory.
For that we consider the residual distance $r_t = \|x_t-x^*\|$ = $\sqrt{h_t}$ and its one-step contraction factor $s_t$, defined as
\[
    s_t\defi\frac{r_{t+1}}{r_t}.
\]
The key input is the tracking estimate in \eqref{eq:tracking-estimate}, which keeps the iterates in the slow-curve regime where the expansions from the previous subsection are valid with controlled remainders.
In \Cref{prop:contraction-ap} we combine these expansions with an exact one-step decrease identity derived from the exact line search, which yields the asymptotic contraction law
\begin{equation}
    \label{eq:contraction-law}
    1-s_t = 2C_p^2\,r_t^\kappa(1+o(1)).
\end{equation}

With this we can finally assemble the lower bound result.
We consider the problem of minimizing the smooth and strongly convex objective $f(x) = \|x-e_1\|^2$ over the $p$-uniform convex set $B_p$.
\gls{fw} with an exact line search initialized at the explicit slow-start point $x_0^{\mathrm{slow}}(u_0)$ from \eqref{eq:slow-start-init} produces a trajectory, for which \Cref{prop:slow-start-tracking} guarantees persistent tracking of the slow curve.
Combined with \eqref{eq:contraction-law}, this yields a precise asymptotic recurrence for the primal gap along that trajectory.
The theorem then follows by summing this asymptotic recurrence, yielding the asymptotic constant and rate.

\begin{restatable}{theorem}{lowerBoundTheorem}\label{thm:lower-bound}\linktoproof{thm:lower-bound}
    For every $p\ge 3$, there exist a $p$-uniformly convex feasible set $\mathcal C \subset\mathbb R^d$ and a smooth and strongly convex function $f:\mathcal C \to\mathbb R$ such that, for a suitable initialization, exact-line-search \gls{fw} generates iterates $(x_t)_{t\geq 0}$ satisfying
    \begin{equation}\label{eq:lower-bound}
        f(x_T)-f(x^*)
        \;\sim\;
    \left(\frac{p+1}{p}\right)^2
    \left(\frac{p}{4(p-1)}\right)^{p/(p-1)}
    T^{-p/(p-1)}
        \qquad\text{as }T\to\infty.
    \end{equation}
    In particular, for any $T \geq 0$, we have $f(x_T)-f(x^*) = \Omega \left(T^{-p/(p-1)} \right)$.
\end{restatable}

This result shows that the order of the convergence rate proved by \citet{kerdreux2021projection} $\bigo{T^{-p/(p-2\theta)}}$ is tight up to constants in the strongly convex case, i.e., for $\theta = \frac12$.
Since the considered quadratic is isotropic, the exact line search and the short step rule coincide.
Thus, the lower bound holds for \gls{fw} with exact line search and for \gls{fw} with short steps.
Moreover, the example is constructed in $\R^2$ but can be extended to arbitrary dimension due to \cref{prop:confinement}.
So unlike the classical results from \citet{nemirovsky1985problem} or the more recent lower bound by \citet{grimmer2026uniform}, our result is not limited to the high-dimensional setting.

The above lower bound is closely related to the one by \citet{halbey2026lower}.
Both approaches consider a two-dimensional projection problem onto the Euclidean ball and the unit $\ell_p$ ball, respectively, and derive a lower bound for a specific initialization.
However, the initialization method as well as the overall proof differ significantly.
For $p=2$, the slow initialization requires very high precision.
Therefore, \citet{halbey2026lower} do not give an explicit formula for the start point $x_0$, but instead choose the end point $x_T$ close to the optimum and then run \gls{fw} backward to obtain a valid slow starting point $x_0$.
In contrast, we give an explicit formula for the initialization in \eqref{eq:slow-start-init}.
Furthermore, we show that the update of the scaled coordinate $y$ is contractive in a neighborhood of the fixed-point map $y^*(u)$.
Thus we prove the robust initialization, which can be observed in \cref{fig:heatmaps}.

Additionally, the proof of the lower-bound rate differs greatly between the two approaches.
\citet{halbey2026lower} derive an explicit recurrence for the distance residual $r_t = \|x_t - e_1\|_2$ and its contraction rate $s_t = \frac{r_{t+1}}{r_t}$ from the \gls{fw} dynamics.
Using that the \gls{fw} update is almost uniquely invertible on the Euclidean ball, they derive bounds on the contraction rates which they ultimately turn into a lower bound for the primal gap.
Our approach considers instead the scaled variables $u$ and $y$ and derives asymptotic expansions of their trajectories.
A crucial ingredient here is the contraction in the $y$ variable which allows us to bound the necessary remainder terms.
Substituting the expansions of $u$ and $y$ in the objective $f$ then yields bounds on the contraction rate and ultimately the primal gap.
Consequently, both approaches use completely different proof techniques, while their model setup and high-level results are quite similar.

Finally, even though the statement in \cref{thm:lower-bound} is only asymptotic, experiments show that the described convergence rate already occurs after a short initial phase.
\Cref{fig:primal-gap-theta-half} compares the slow-start trajectory to additional exact-line-search runs from generic initializations on $B_p$ for $p\in\{3,5\}$ in the strongly convex case $\theta=\tfrac12$, together with a scaled reference curve $T^{-p/(p-1)}$.
The slow initialization aligns with this slope after a short transient, illustrating that the asymptotic rate is roughly visible already for $t \geq 100$.
Finally, the generic initializations converge faster than the slow asymptotic rate.

\begin{figure}[ht]
    \centering
    \includegraphics[width=\linewidth]{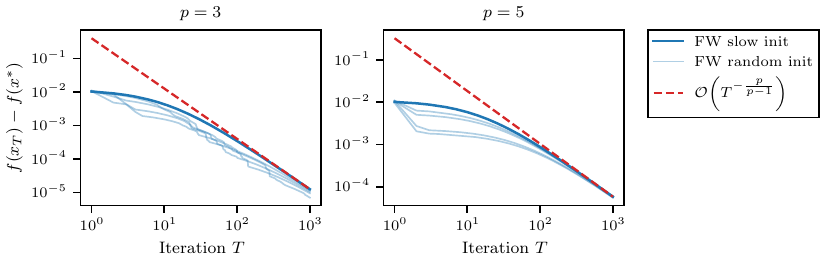}
    \caption{Comparison of the primal gap $f(x_T)-f(x^*)$ versus iteration $T$ for exact-line-search \gls{fw} on the minimization problem in \eqref{eq:model} for $p \in \{3,5\}$.
    The solid curve uses the slow initialization \eqref{eq:slow-start-init}, while the faint curves are runs from generic initializations.
    The dashed line shows $T^{-p/(p-1)}$, which matches both our lower bound from \cref{thm:lower-bound} and the upper bound rate of \citet{kerdreux2021projection}.}
    \label{fig:primal-gap-theta-half}
\end{figure}

%% ====================================================================
\section{Extension to HEB functions}\label{sec:sharp-extension}
%% ====================================================================

The goal of this section is to establish a matching lower bound counterpart to the general \gls{heb}-rate upper bound of \citet{kerdreux2021projection}.
Up to now, our explicit lower bound construction is for strongly convex objectives, i.e., the \gls{heb} case $\theta=\frac12$ in \Cref{thm:lower-bound}.

We transfer the previous model to general $\theta\in(0,\frac12]$ by applying a power transform to the objective in \eqref{eq:model}.
For any target \gls{heb} constant $\mu>0$, we consider
\begin{equation}
    \label{eq: def g}
    g(x)\defi \mu^{-1/\theta}\|x-e_1\|_2^{1/\theta}
    = \mu^{-1/\theta} f(x)^{1/(2\theta)}.
\end{equation}
The new objective $g$ is smooth on $\mathcal{C}$ for $\theta\in(0,\frac12]$ and satisfies the $(\mu,\theta)$-\gls{heb}.
This generalization of the objective $f$ yields a very elegant extension, as the \gls{fw} trajectories for $f$ and $g$ on $B_p$ coincide (see \Cref{prop:naive-power-extension}).
The gradient of $g$ contains only a positive scalar scaling of $\nabla f$, which does not affect the \gls{lmo} output.
Furthermore, since $y \mapsto y^{1/(2\theta)}$ is monotone increasing, exact line search produces the same step sizes for $g$ as for $f$.
Importantly, the trajectories only coincide under exact line search: for the transformed objective $g$, short step and exact line search no longer agree when $\theta<\tfrac12$.

Due to the coinciding trajectories, one can directly reparameterize the asymptotics in \cref{thm:lower-bound} to obtain convergence rates for $g$ starting from the slow initialization in \eqref{eq:slow-start-init}.

\begin{restatable}{theorem}{extendedLowerBound}\label{thm:extended-lower-bound}\linktoproof{thm:extended-lower-bound}
    For every $p\ge 3$, $\mu>0$ and $\theta \in (0, \frac12]$, there exist a $p$-uniformly convex feasible set $\mathcal C \subset\mathbb R^d$ and a smooth function $g:\mathcal C \to\mathbb R$ satisfying the $(\mu, \theta)$-\gls{heb} such that, for a suitable initialization, exact-line-search \gls{fw} generates iterates $(x_t)_{t\geq 0}$ satisfying
    \begin{equation}\label{eq:extended-lower-bound}
        g(x_T)-g(x^*)
        \;\sim\;
    \mu^{-\frac{1}{\theta}}\left(\frac{p+1}{p}\right)^{\frac{1}{\theta}}
    \left(\frac{p}{4(p-1)}\right)^{\frac{p}{2\theta(p-1)}}
    T^{-\frac{p}{2\theta(p-1)}}
        \qquad\text{as }T\to\infty.
    \end{equation}
    In particular, for any $T \geq 0$, we have $g(x_T)-g(x^*) = \Omega \left(T^{-\frac{p}{2\theta(p-1)}} \right)$.
\end{restatable}
This yields a lower bound for \gls{fw} under $(\mu,\theta)$-\gls{heb} for $p$-uniformly convex sets with $p \geq 3$, in arbitrary dimension $d\ge 2$, using the same slow initialization defined in \eqref{eq:slow-start-init}.
For $\theta=\tfrac12$, it recovers \cref{thm:lower-bound} and matches the upper bound of \citet{kerdreux2021projection}.
For $\theta\in(0,\tfrac12)$, the exponent $\frac{p}{2\theta(p-1)}$ does not match the known upper-bound exponent $\frac{p}{p-2\theta}$.
It remains open whether that upper bound is tight, or whether a faster worst-case guarantee is possible in this regime.

In \Cref{fig:primal-gap-heb} we compare the primal gap $g(x_t)-g(x^*)$ along exact-line-search \gls{fw} trajectories on $B_p\subset\R^2$ starting from the slow initialization in \eqref{eq:slow-start-init} with the upper bound rate by \citet{kerdreux2021projection} and our new derived lower bound in \cref{thm:extended-lower-bound}.
We compare these quantities for $p\in\{3,5\}$ and $\theta\in\{\tfrac13,\tfrac14\}$.
In each setting we observe that the numerics match our results in \cref{thm:extended-lower-bound}.
After a short transient the primal gap tracks the reference line $T^{-p/(2\theta(p-1))}$ representing the derived lower bounds in \eqref{eq:extended-lower-bound}.
The reference line for $T^{-p/(p-2\theta)}$ demonstrates the gap to the upper bound by \citet{kerdreux2021projection} which manifests as a shallower slope than the primal gap.

\begin{figure}[t]
    \centering
    \includegraphics[width=\linewidth]{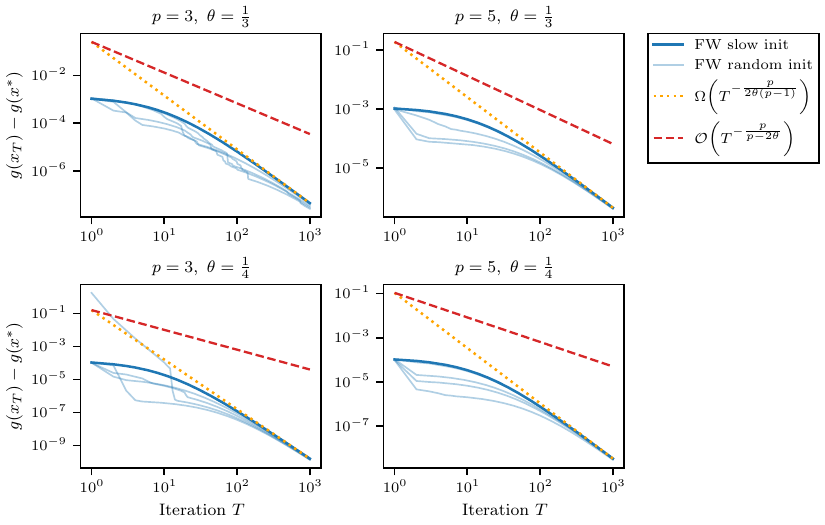}
    \caption{Comparison of the primal gap $g(x_T)-g(x^*)$ versus iteration $T$ for exact-line-search \gls{fw} for minimizing \eqref{eq: def g} on $B_p$ for $p\in\{3,5\}$ and $\theta\in\{\tfrac13,\tfrac14\}$.
    The solid curve uses the slow initialization \eqref{eq:slow-start-init}, while the faint curves are runs from generic initializations.
    The dotted reference line shows the $\Omega\left(T^{-p/(2\theta(p-1))}\right)$ lower bound from \cref{thm:extended-lower-bound}, while the dashed line shows the $\bigo{T^{-p/(p-2\theta)}}$ upper bound by \citet{kerdreux2021projection}.
    Both reference lines are scaled for clearer visualization.}
    \label{fig:primal-gap-heb}
\end{figure}

%% ====================================================================
\section{Conclusion}\label{sec:discussion}
%% ====================================================================

We established matching $\Omega \big(T^{-p/(p-1)}\big)$ lower bounds on the convergence of \gls{fw} with exact line search or short steps over $p$-uniformly convex feasible sets for every $p\ge 3$, pairing with the upper bounds of \citet{kerdreux2021projection}.
The arguments follow the dynamics of \gls{fw} on explicit instances on $\ell_p$ balls and apply in any fixed dimension $d \geq 2$, unlike typical information-theoretic constructions.
We further extended this framework to smooth convex objectives satisfying a Hölderian error bound, recovering the uniformly convex exponent when $\theta=\tfrac12$ while leaving open whether the known upper bounds are tight for $\theta\in(0,\tfrac12)$.

\paragraph{Open questions.}
There are several natural questions left open by our analysis.
First, our lower bound in the \gls{heb} case does not match the known upper bound by \citet{kerdreux2021projection}.
It remains open whether one can find a matching lower bound or if one can prove a faster convergence guarantee.
Second, our lower bounds are stated for \gls{fw} with exact line search or short steps.
\citet{wirth23acceleration} show that open-loop step-size rules can converge faster on uniformly convex sets in some regimes, and a matching lower bound for open-loop \gls{fw} would close this gap. 
%Next, there do not exist any lower bounds for the case of $p \in (2,3)$.
Finally, \gls{fw} variants such as away-step and pairwise \gls{fw} are known to accelerate the algorithm on polytopes.
Whether away-step and pairwise \gls{fw} can bypass our lower bound on uniformly convex sets is still unknown.

\paragraph{AI-assisted proof discovery.}
The framework of \citep{zimmer2026agentic} couples a general-purpose command-line coding agent with a sandboxed workspace, a project-specific instruction file, and a set of methodological rules (``commandments'') that enforce single-variable experimentation, staged verification, and structured reporting.
The human researcher steers the agent at the level of subtasks rather than individual edits, and intermediate artefacts (proof drafts, numerical checks, plots) are version-controlled so the trajectory of the project remains inspectable.

We gave the agent the prior strongly convex lower bounds and the matching $p$-uniformly convex upper bound (\Cref{tab:rates}) and asked it to find a tight lower bound in the uniformly convex regime.
It first attempted to generalize the high-dimensional construction of \citet{grimmer2026uniform} to $\ell_p$ balls.
This failed, as the construction relies on decomposing strongly convex sets as intersections of shifted Euclidean balls and the agent did not find a direct analogue for $p>2$.
The agent then pivoted to the dynamics-based approach of \citet{halbey2026lower}, derived the \gls{fw} iteration on $\ell_p$ balls in closed form, and verified each component numerically with Julia \texttt{BigFloat} scripts, totaling more than thirty individual checks.
Experiments across multiple values of $p$ revealed that the iterates alternate in sign and settle onto a low-dimensional invariant curve whose shape can be characterized analytically, which suggested the right proof strategy.
The agent first estimated the relevant constants numerically and only then derived them in closed form, before assembling a rigorous proof for $p\ge 3$.
The boundary case $p\in(2,3)$, where the sign alternation is intermittent, was identified empirically as qualitatively different and currently lies outside the reach of our technique.
Once the agent had produced a complete draft, we proofread the proof carefully, improved the presentation, and fixed minor issues.
Finally, the extension to \gls{heb} functions in \Cref{sec:sharp-extension} was derived by hand.

\begin{ack}
    Funded by the Deutsche Forschungsgemeinschaft (DFG, German Research Foundation) under Germany's Excellence Strategy – The Berlin Mathematics Research Center MATH+ (EXC-2046/1, EXC-2046/2, project ID: 390685689).
\end{ack}

\bibliographystyle{plainnat}
\bibliography{refs_arixv}

%%%%%%%%%%%%%%%%%%%%%%%%%%%%%%%%%%%%%%%%%%%%%%%%%%%%%%%%%%%%
\newpage
\appendix

\section{Proofs}\label{sec:proofs}

\begin{lemma}\label{lem:lmo-lp}
    Let $p\in(1,\infty)$, $q=p/(p-1)$, $g\in\R^d \setminus \{0\}$, and define
    $v\in\R^d$ by
    \[
        v_i = -\frac{\operatorname{sign}(g_i)\,|g_i|^{q-1}}{\norm{g}_q^{q-1}},
    \]
    for all $i=1,\dots,d$. Then
    \[
        v=\arg\min_{w\in B_p}\ip{g,w}.
    \]
\end{lemma}
\begin{proof}
    By Hölder we have,
    \[
      \ip{g,w}\ge -\norm{g}_q\norm{w}_p\ge -\norm{g}_q
    \]
    for all $w\in B_p$.
    Hence the minimum value is at least $-\norm{g}_q$. Using the definition of $v$
    and $(q-1)p=q$,
    \[
        \norm{v}_p^p
        =\sum_i \frac{|g_i|^{(q-1)p}}{\norm{g}_q^{(q-1)p}}
        =\frac{\sum_i |g_i|^q}{\norm{g}_q^q}
        =1,
    \]
    so $v\in B_p$. Moreover,
    \[
        \ip{g,v}
        =-\sum_i \frac{|g_i|^q}{\norm{g}_q^{q-1}}
        =-\norm{g}_q.
    \]
    Thus $v$ attains the lower bound and is optimal.
    Since the objective is linear and $B_p$ is strictly convex, the minimizer is unique.
\end{proof}

\confinementProposition*
\begin{proof}\linkofproof{prop:confinement}
    If $x\in\spann\{e_1,e_2\}$ then only coordinates 1 and 2 of $\nabla f(x)=2(x-e_1)$ are nonzero.
    \cref{lem:lmo-lp} shows that the LMO on $B_p$ acts componentwise in the dual power
    $|\cdot|^{q-1}$, so the minimizer $v$ has nonzero entries only in coordinates 1
    and 2. The update preserves the span.
\end{proof}

\begin{proposition}\label{prop:uw-dynamics}
    Let $q=p/(p-1)$ and $M\defi (u^q+|w|^q)^{1/q}$.
    One \gls{fw} step maps $(u,w)\mapsto(u',w')$ via
    \begin{align}
        v_1 &= \frac{u^{q-1}}{M^{q-1}}, &
        v_2 &= -\frac{|w|^{q-1}}{M^{q-1}}\,\sign(w), \label{eq:lmo-uw} \\
        d_1 &= v_1-1+u, &
        d_2 &= v_2-w, \label{eq:dir-uw} \\
        \gamma &= \frac{M-u+u^2+w^2}{d_1^2+d_2^2}, \label{eq:gamma-uw} \\
        u' &= u-\gamma d_1, &
        w' &= w+\gamma d_2. \label{eq:update-uw}
    \end{align}
\end{proposition}
\begin{proof}\linkofproof{prop:uw-dynamics}
    Writing the objective in terms of $(u,w)$ gives $f(x) = \norm{x-e_1}^2 = u^2+w^2$.
    Using $\nabla f=(-2u,2w)^\top$ with $\norm{\nabla f}_q=2(u^q+|w|^q)^{1/q} = 2M$ and \cref{lem:lmo-lp} yields
    \begin{align*}
        v_1 &= -\frac{\sign(-u)\,|u|^{q-1}}{M^{q-1}} = \frac{u^{q-1}}{M^{q-1}}\\
        v_2 &= -\frac{\sign(w)\,|w|^{q-1}}{M^{q-1}},
    \end{align*}
    where we have used that $u > 0$ since $x_1 < 1$.
    The update direction is then given by $d_1 = v_1 - x_1 = v_1 -1+u$ and $d_2 = v_2 - x_2 = v_2 - w$.
    The exact line search minimizing
    \begin{align*}
        f(x+\gamma(v-x))
        &= f(x)+\gamma\ip{\nabla f, v-x}+\gamma^2\norm{v-x}^2 \\
        &= u^2+w^2+2\gamma(-ud_1+wd_2)+\gamma^2(d_1^2+d_2^2).
    \end{align*}
    yields
    \begin{equation}
        \label{eq:gamma-uw-expand}
        \gamma = \frac{ud_1-wd_2}{d_1^2+d_2^2} 
        = \frac{u(v_1-1+u) - w(v_2-w)}{d_1^2+d_2^2}
        = \frac{M-u+u^2+w^2}{d_1^2+d_2^2},
    \end{equation}
    where we have used that $uv_1-wv_2 = \frac{u^q+|w|^q}{M^{q-1}} = M$.
    The update \eqref{eq:update-uw} follows immediately,
    \begin{align*}
        u' &= 1 - x_1' = 1 -(x_1 + \gamma d_1) = u - \gamma d_1\\
        w' &= x_2' = x_2 + \gamma d_2 = w + \gamma d_2.
    \end{align*}
\end{proof}

\begin{corollary}\label{cor:ratio-identity}
Assume $w\neq 0$ and define
$P \defi  |w|(1-v_1)+u|v_2|>0$. Then
\[
    u'=\frac{|d_2|\,P}{d_1^2+d_2^2},
    \qquad
    |w'|=\frac{|d_1|\,P}{d_1^2+d_2^2},
    \qquad
    \rho'=\frac{|w'|}{u'}=\frac{|d_1|}{|d_2|}.
\]
In particular, if $d_1\neq 0$ then $w'\neq 0$.
\end{corollary}
\begin{proof}\linkofproof{cor:ratio-identity}
    A direct expansion of \eqref{eq:update-uw} using the first identity in \eqref{eq:gamma-uw-expand} gives
    \begin{align}
        u' &= u-\gamma d_1
        =\frac{u(d_1^2 + d_2^2) - ud_1^2+ w d_1 d_2}{d_1^2+d_2^2}
        = d_2 \frac{w d_1 + u d_2}{d_1^2+d_2^2}\label{eq:u-update-expand}\\
        w' &= w+\gamma d_2
        = \frac{w(d_1^2 + d_2^2) + u d_1 d_2 - w d_2^2}{d_1^2+d_2^2}
        = d_1 \frac{w d_1 + u d_2}{d_1^2+d_2^2}\label{eq:w-update-expand}.
    \end{align}
    Furthermore, we have
    \[
        w d_1 + u d_2 = w(v_1-1+u) + u (v_2-w) = w(v_1-1) + u v_2.
    \]
    Since $v_1 < 1$, $u>0$ and $\sign(v_2) = -\sign(w)$, we have $\sign(w(v_1-1)) = - \sign(uv_2)$ and thus
    \[
        |w d_1 + u d_2| = |w(v_1-1) + u v_2| = |w|(1-v_1)+u|v_2|=P.
    \]
    Taking the absolute values of \eqref{eq:u-update-expand} and \eqref{eq:w-update-expand} and using $u' = 1- x_1' > 0$ yields the stated results.
    If $d_1\neq 0$ then $|w'|>0$ since $P>0$.
    Finally, we have $\rho'=\frac{|w'|}{u'} = \frac{|d_1|}{|d_2|}$.
\end{proof}

\begin{lemma}\label{lem:y-map}\linktoproof{lem:y-map}
    Fix $p\ge 3$ and let $q=p/(p-1)$, $\alpha=(p-1)/p$, and $\kappa=2\alpha$.
    Let $I\subset(0,p^{1/q})$ be compact.
    Then uniformly for $y\in I$, as $u\downarrow 0$,
    \begin{equation}\label{eq:y-map}
        \Phi(u,y)
        = F(y)+uG(y)+O(u^\kappa),
        \qquad
        F(y)=y^{-(q-1)}-\frac{y}{p},
        \qquad
        G(y)=\frac{y}{p}+\frac{p-1}{2p^2}\,y^{q+1}.
    \end{equation}
    Moreover, uniformly on $I$,
    \[
        \frac{\partial}{\partial y}\Phi(u,y)=F'(y)+uG'(y)+O(u^\kappa),
        \qquad
        \frac{\partial}{\partial u}\Phi(u,y)=G(y)+O(u^{\kappa-1}).
    \]
\end{lemma}
\begin{proof}\linkofproof{lem:y-map}
    The proof is divided into three parts.
    First, we are going to rewrite $\Phi(u,y)$ using some auxiliary functions that depend solely on $u$ and $y$ (or $z = y^q$).
    In the second part, we will expand each factor of the derived identity of $\Phi(u,y)$ around $u=0$.
    Finally, we show how these expansions extend to the partial derivatives of $\Phi(u,y)$.

    We begin by rewriting $\Phi(u,y)$ using \cref{cor:ratio-identity}
    \begin{equation}\label{eq:yprime-ratio}
        \Phi(u,y)= y'
        = \frac{\rho'}{(u')^\alpha}
        = \frac{|d_1|}{|d_2|(u')^\alpha}.
    \end{equation}
    In the following steps we will rewrite $d_1$, $d_2$ and $u'$ in terms of $u$ and $y$.
    First we define the auxiliary variable $A(u,y)\defi (1+y^q u)^{\frac1p}$.
    For any compact $I \subset (0,p^{1/q})$ and any $y \in I$, we have $z = y^q < p$.
    Using $\log(1+x) < x$ for $x > 0$, we have
    \[
        \log A(u,y) < \log (1+p u)^{\frac1p} = \frac1p \log (1+p u) < \frac1p (p u) = u,
    \]
    and therefore $\frac{1}{A(u,y)} > e^{-u}$.
    Next note that
    \begin{align}
        \label{eq:A-expand}
        A(u,y)&\defi (1+y^q u)^{\frac1p}
        = \left( 1 + \frac{|w|^q}{u^q}\right)^{\frac1p}
        =\frac{\left(u^q + |w|^q\right)^{\frac1p}}{u^{\frac qp}} \notag\\
        &= \left(\frac{M}{u}\right)^{\frac qp} = \left(\frac{M}{u}\right)^{q-1} = \frac{1}{v_1}.
    \end{align}
    Thus, by \eqref{eq:dir-uw} we have
    \begin{equation}
        \label{eq:d1-expand}
        d_1 = u + v_1 - 1 = u + \frac{1}{A} - 1 > u + e^{-u} - 1 > 0.
    \end{equation}
    For the second coordinate, we use $A = \left( \frac{M}{u}\right)^{q-1}$ from \eqref{eq:A-expand} to get
    \[
        |v_2| = \left(\frac{|w|}{M}\right)^{q-1} = \left(\frac{\rho u}{M}\right)^{q-1} = \frac{\rho^{q-1}}{A}
    \]
    and thus
    \begin{equation}
        \label{eq:d2-abs-expand}
        |d_2| = |v_2 - w|
        = |v_2|+|w|
        = \frac{\rho^{q-1}}{A} + y u^{1+\alpha}
        = \frac{(y u^\alpha)^{q-1}}{A} + y u^{1+\alpha}
        = \frac{y^{q-1}u^{\frac 1p}}{A} + y u^{1+\alpha}
    \end{equation}
    where we have used $\sign(v_2) = -\sign(w)$ and $\alpha(q-1)=\frac{q-1}{q}=\frac{1}{p}$.
    Next we define the auxiliary variables $E(u,y)$ and $\eta(u,y)$
    \[
        E(u,y)\defi y^{2-q}u^\kappa A(u,y),
        \qquad
        \eta(u,y)\defi \frac{\gamma d_1}{u}.
    \]
    Then we have
    \[
        |d_2|\,u^\alpha = \frac{y^{q-1}u}{A}\,(1+E),
        \qquad
        u'=u(1-\eta),
    \]
    so map $\Phi(u,y)$ admits the exact factorization
    \begin{equation}\label{eq:phi-factorization}
        \Phi(u,y) = \frac{d_1}{|d_2|(u')^\alpha}
        = \frac{d_1}{|d_2|u^\alpha(1-\eta)^\alpha}
        = \frac{A}{y^{q-1}}\frac{d_1}{u}(1+E)^{-1} (1-\eta)^{-\alpha},
    \end{equation}
    where we have dropped the absolute value for $d_1$ since $d_1 > 0$.

    Define
    \[
        B(u,y)\defi \frac{A}{y^{q-1}}\frac{d_1}{u},
        \qquad
        R(u,y)\defi (1+E)^{-1}(1-\eta)^{-\alpha},
    \]
    so that $\Phi=B\cdot R$ by \eqref{eq:phi-factorization}.
    We now expand the factors of $B$ and $R$ around $u=0$. By compactness of $I$,
    all estimates below are uniform for $y\in I$.
    After shrinking the $u \downarrow 0$ threshold, we may assume $zu=y^q u \leq \frac12$ for all $y \in I$.

    First we consider $A(u,y) = (1+zu)^{1/p}$.
    Since $A(0,y) = 1$ for any $y \in I$ and
    \begin{equation}
        \left.\frac{\partial }{\partial u} A(u,y)\right|_{u=0}
        = \left.\frac{1}{p} (1+zu)^{\frac 1p -1}z\right|_{u=0}
        = \frac{z}{p},
        \label{eq:A-deriv-1}
    \end{equation}
    we have
    \begin{equation}
        \label{eq:A-u-expand}
        A(u,y) = 1 + \frac{z}{p} u + \mathcal{O}(u^2).
    \end{equation}

    Next we compute the derivatives of $\frac{1}{A(u,y)}$ for $u=0$
    \begin{align}
        \left.\frac{\partial }{\partial u} \frac{1}{A(u,y)}\right|_{u=0}
        &= \left. -\frac{\frac{\partial }{\partial u} A(u,y)}{A(u,y)^2} \right|_{u=0}
        = -\frac{z}{p},
        \label{eq:invA-deriv-1}
        \\
        \left.\frac{\partial^2 }{\partial u^2} \frac{1}{A(u,y)}\right|_{u=0}
        &= \frac{2 \left(\frac{\partial }{\partial u} A(u,y) \right)^2 - A(u,y) \cdot \frac{\partial^2 }{\partial u^2} A(u,y)}{A(u,y)^3} \bigg|_{u=0} \notag \\
        &= 2 \left(\frac{z}{p}\right)^2 - 1 \cdot \frac{1-p}{p^2}z^2
        = \frac{p+1}{p^2}z^2,
        \label{eq:invA-deriv-2}
    \end{align}
    where we have used $A(0,y) = 1$ and \eqref{eq:A-deriv-1} as well as
    \begin{equation*}
        \left.\frac{\partial^2 }{\partial u^2} A(u,y)\right|_{u=0}
        = \left.\frac{1}{p}\left(\frac{1}{p}-1\right)(1+zu)^{\frac{1}{p}-2} z^2\right|_{u=0}
        = \frac{1-p}{p^2}z^2.
        \label{eq:A-deriv-2}
    \end{equation*}
    Thus, the expansion of $\frac{1}{A(u,y)}$ for $u=0$ yields
    \begin{align}
        \label{eq:d1-over-u-expand}
        \frac{d_1}{u}
        &= \frac{u + \frac{1}{A} - 1}{u}
        = \frac{u  + \left( 1 - \frac{z}{p} u + \frac12 \frac{p+1}{p^2}z^2 u^2 + O(u^3) \right) - 1}{u} \notag \\
        &= 1-\frac{z}{p} + \frac{p+1}{2p^2}z^2 u + O(u^2).
    \end{align}
    Combining \eqref{eq:A-u-expand} and \eqref{eq:d1-over-u-expand} yields
    \begin{align}
        B(u,y)
        &= \frac{A}{y^{q-1}}\frac{d_1}{u} \notag\\
        &= y^{-(q-1)} \left( 1 + \frac{z}{p} u + \mathcal{O}(u^2) \right) \left( 1 - \frac{z}{p} + \frac{p+1}{2p^2}z^2 u + O(u^2) \right) \notag\\
        &= F(y) + uG(y) + O(u^2). \label{eq:B-u-expand}
    \end{align}
    Next we consider the terms of $R(u,y)$ starting with $(1+E)^{-1}$.
    Using the expansion $(1+x)^{-1} = 1 - x + O(x^2)$ for $x \to 0$ and
    \[
        E(u,y) = y^{2-q}u^\kappa A(u,y) = y^{2-q}u^\kappa \left( 1 + \frac{z}{p} u + \mathcal{O}(u^2) \right) = \mathcal{O}(u^\kappa),
    \]
    we get
    \begin{equation}
      \label{eq:1+E-inv-expand}
        (1+E)^{-1} = 1 + O(u^\kappa).
    \end{equation}
    In the next step we expand $(1-\eta)^{-\alpha}$.
    For that we rewrite $\eta$ using \eqref{eq:gamma-uw},
    \[
        \eta = \frac{\gamma d_1}{u} = \frac{d_1}{u} \cdot \frac{M - u + h}{d_1^2+d_2^2}.
    \]
    First we see that
    \[
        u(1+zu)^\alpha = u \left( 1 + \frac{|w|^q}{u^q}\right)^{\frac 1q} = (u^q + |w|^q)^{\frac 1q} = M,
    \]
    and
    \[
        h = u^2 + w^2 = u^2 + y^2 u^{2(1+\alpha)} = u^2 + y^2 u^{2+\kappa}.
    \]
    Expansion yields $(1+zu)^\alpha = 1 + \alpha z u + O(u^2)$ and thus we get
    \begin{align}
        \label{eq:M-u+h-expand}
        M-u+h
        &= u(1+zu)^\alpha - u + u^2 + y^2 u^{2+\kappa} \notag\\
        &= u(1+\alpha z u + \mathcal{O}(u^2)) - u + u^2 + y^2 u^{2+\kappa} = \mathcal{O}(u^2).
    \end{align}
    Furthermore, by \eqref{eq:d1-over-u-expand} we have $d_1 = \mathcal{O}(u)$
    and \eqref{eq:d2-abs-expand} and \eqref{eq:A-u-expand} imply $|d_2| = \Theta(u^{1/p})$.
    Therefore, together with \eqref{eq:M-u+h-expand},
    \begin{equation}
        \label{eq:gamma-expand}
        \gamma = \frac{M - u + h}{d_1^2+d_2^2} = \frac{\mathcal{O}(u^2)}{\Theta(u^{2/p})} = \mathcal{O}(u^{2- 2/p}) = \mathcal{O}(u^\kappa).
    \end{equation}
    Together with \eqref{eq:d1-over-u-expand} this yields
    \[
        \eta = \frac{d_1}{u} \cdot \gamma = \mathcal{O}(1) \cdot \mathcal{O}(u^\kappa) = \mathcal{O}(u^\kappa).
    \]
    With $(1-x)^{-\alpha} = 1 + \alpha x + O(x^2)$ for $x \to 0$ we get
    \begin{equation}
      \label{eq:1-eta-alpha-expand}
        (1-\eta)^{-\alpha} = 1 + \alpha \eta + O(\eta^2) = 1 + \mathcal{O}(u^\kappa).
    \end{equation}
    Therefore we have
    \begin{equation}
        \label{eq:R-u-expand}
        R(u,y)= (1+E)^{-1}(1-\eta)^{-\alpha} = (1+O(u^\kappa))(1+O(u^\kappa)) = 1+O(u^\kappa).
    \end{equation}
    Combining with \eqref{eq:B-u-expand} yields
    \begin{align}
        \label{eq:phi-u-expand}
        \Phi(u,y) &= B(u,y)R(u,y) = (F(y) + uG(y) + O(u^2))(1+O(u^\kappa)) \notag\\
        &= F(y) + uG(y) + O(u^2) + O(u^\kappa) = F(y) + uG(y) + O(u^\kappa),
    \end{align}
    as $\kappa = 2\cdot \frac{p-1}{p} < 2$.

    In the last step, we derive expansions for the partial derivatives of $\Phi(u,y)$ starting with $B$.
    First note that $A(u,y)=(1+y^q u)^{1/p}$ is smooth and $A(u,y) > 0$ for all $(u,y) \in [0,u_0]\times I$
    and thus $d_1 = u + \frac 1A - 1$ is smooth.
    Furthermore, since $d_1(0,y) = 0$ for all $y \in I$, the quotient $\frac{d_1}{u}$ is also smooth and finally $B(u,y)=y^{-(q-1)}A(u,y)\frac{d_1}{u}$ is smooth.
    Since $[0,u_0]\times I$ is compact, all second partial derivatives of $B(u,y)=F(y)+uG(y)+O(u^2)$ are bounded there
     and thus the $O(u^2)$ remainder stays bounded after differentiation.
    This yields
    \[
        \frac{\partial}{\partial y} B(u,y)=F'(y)+uG'(y)+O(u^2),\qquad
        \frac{\partial}{\partial u} B(u,y)=G(y)+O(u).
    \]
    uniformly for $y \in I$.

    Next we estimate derivatives of $R=(1+E)^{-1}(1-\eta)^{-\alpha}$.
    Since
    \[
        E(u,y)=y^{2-q}u^\kappa A(u,y),
    \]
    with $A=O(1)$, $\frac{\partial}{\partial y} A=O(u)$ and $\frac{\partial}{\partial u} A=O(1)$ uniformly on $I$, product
    rule gives
    \[
      \frac{\partial}{\partial y} E(u,y)=O(u^\kappa),\qquad
      \frac{\partial}{\partial u} E(u,y)=O(u^{\kappa-1}).
    \]
    Since $A >0$, we have $E(u,y) \geq 0$ and thus $(1+E)^{-2} = \mathcal{O}(1)$.
    The chain rule yields then
    \[
        \frac{\partial}{\partial y} (1+E)^{-1}=-(1+E)^{-2}\frac{\partial}{\partial y} E=O(u^\kappa),\qquad
        \frac{\partial}{\partial u} (1+E)^{-1}=-(1+E)^{-2}\frac{\partial}{\partial u} E=O(u^{\kappa-1}).
    \]
    Next we write $\eta=\frac{d_1}{u}\gamma$ with
    \[
        \gamma=\frac{M-u+h}{d_1^2+d_2^2},
    \]
    Since $M$ and $h$ are both twice continuously differentiable with $M - u + h  = \mathcal{O}(u^2)$, we get
    \[
        \frac{\partial}{\partial y} (M-u+h)=O(u^2),\qquad
        \frac{\partial}{\partial u} (M-u+h)=O(u).
    \]
    Further, $d_1=O(u)$ and \eqref{eq:d2-abs-expand} imply
    \[
        d_1^2+d_2^2=\Theta(u^{2/p}),\quad
        \frac{\partial}{\partial y}(d_1^2+d_2^2)=O(u^{2/p}),\quad
        \frac{\partial}{\partial u}(d_1^2+d_2^2)=O(u^{2/p-1}).
    \]
    Applying the quotient rule to $\gamma$ gives
    \[
        \frac{\partial}{\partial y} \gamma=O(u^\kappa),\qquad
        \frac{\partial}{\partial u} \gamma=O(u^{\kappa-1}),
    \]
    since $\kappa=2-2/p$.
    Furthermore, \eqref{eq:d1-over-u-expand} and the smoothness of $\frac{d_1}{u}$ yield
    \[
        \frac{\partial}{\partial y} \left(\frac{d_1}{u}\right)=O(1),\qquad
        \frac{\partial}{\partial u} \left(\frac{d_1}{u}\right)=O(1).
    \]
    Therefore, by product rule for $\eta=\frac{d_1}{u}\gamma$ we get
    \[
        \frac{\partial}{\partial y} \eta=O(u^\kappa),\qquad
        \frac{\partial}{\partial u} \eta=O(u^{\kappa-1}).
    \]
    For $u_0$ small enough we have $|\eta|\le \frac12$, hence
    $(1-\eta)^{-\alpha-1}=O(1)$ uniformly, and chain rule yields
    \[
        \frac{\partial}{\partial y} (1-\eta)^{-\alpha}=O(u^\kappa),\qquad
        \frac{\partial}{\partial u} (1-\eta)^{-\alpha}=O(u^{\kappa-1}).
    \]
    Combining with the bounds for $(1+E)^{-1}$ and its derivatives,
    \[
        \frac{\partial}{\partial y} R(u,y)=O(u^\kappa),\qquad
        \frac{\partial}{\partial u} R(u,y)=O(u^{\kappa-1}).
    \]
    Finally, from $\Phi=B\cdot R$, we get 
    \[
        \frac{\partial}{\partial y} \Phi(u,y)=(\frac{\partial}{\partial y} B)R+B\,\frac{\partial}{\partial y} R,\qquad
        \frac{\partial}{\partial u} \Phi(u,y)=(\frac{\partial}{\partial u} B)R+B\,\frac{\partial}{\partial u} R.
    \]
    Since $B=F+uG+O(u^2)=O(1)$ and all bounds are uniform on $I$, we get
    \[
        \frac{\partial}{\partial y} \Phi(u,y)=F'(y)+uG'(y)+O(u^\kappa),\qquad
        \frac{\partial}{\partial u} \Phi(u,y)=G(y)+O(u^{\kappa-1}).
    \]
    This completes the proof.
\end{proof}

\begin{proposition}\label{prop:z-drift}\linktoproof{prop:z-drift}
    Fix $p\ge 3$ and let $\Phi$ be as above.
    There exist $u_0>0$ and a function $y^*:[0,u_0]\to(0,\infty)$,
    continuous on $[0,u_0]$, $C^1$ on $(0,u_0]$, and Lipschitz on $[0,u_0]$, with
    $y^*(0)=C_p$ such that for all $u\in(0,u_0]$,
    \[
        \Phi(u,y^*(u)) = y^*(u).
    \]
    Moreover, as $u\downarrow 0$,
    \[
        y^*(u)= C_p + D_p u + O(u^\kappa),
        \qquad
        C_p = \left(\frac{p}{p+1}\right)^{1/q},
        \qquad
        D_p = C_p\cdot \frac{(p-1)(3p+1)}{2p(p+1)^2},
    \]
    and with $z^*(u) \defi y^*(u)^q$,
    \begin{equation}\label{eq:z-drift}
        z^*(u)=\frac{p}{p+1}+\frac{p(3p+1)}{2(p+1)^3}\,u+o(u).
    \end{equation}
\end{proposition}

\begin{proof}\linkofproof{prop:z-drift}
    Choose a closed interval $J=[C_p-\delta_0,C_p+\delta_0]\subset(0,p^{1/q})$.
    By \Cref{lem:y-map}, after possibly shrinking $u_0>0$ there exist constants
    $C_0,C_1,C_2>0$ and remainder functions $R_0,R_1,R_2$ on
    $(0,u_0]\times J$ such that
    \begin{align}
        H(u,y)\defi \Phi(u,y)-y
        &= F(y)-y+uG(y)+R_0(u,y), \label{eq:H-expand}\\
        \frac{\partial}{\partial y} H(u,y)
        &= F'(y)-1+uG'(y)+R_1(u,y), \label{eq:Hy-expand}\\
        \frac{\partial}{\partial u} H(u,y)
        &= G(y)+R_2(u,y), \label{eq:Hu-expand}
    \end{align}
    with
    \begin{equation}
        \label{eq:R-bounds}
        |R_0(u,y)|\le C_0u^\kappa,
        \qquad
        |R_1(u,y)|\le C_1u^\kappa,
        \qquad
        |R_2(u,y)|\le C_2u^{\kappa-1}
    \end{equation}
    for all $(u,y)\in(0,u_0]\times J$.
    The function $F(y) = y^{1-q} - \frac{y}{p}$ has two fixed points at $y=0$ and $y=C_p$, 
    \[
        F(y) = y^{1-q}-\frac{y}{p}=y \;\Longleftrightarrow\;
        y^{1-q}=\frac{p+1}{p}\,y \;\Longrightarrow\;
        y^q=\frac{p}{p+1} \;\Longleftrightarrow\;
        y=\left(\frac{p}{p+1}\right)^{1/q} = C_p.
    \]
    In particular, $F$ has a unique fixed point on $J \subset (0,p^{1/q})$.
    Moreover,
    \begin{align*}
    F'(C_p)-1
        % &=(1-q)C_p^{-q}-\frac{1}{p}-1\\
        % &=(1-q)\frac{p+1}{p}-\frac{1}{p}-1\\
        % &=-\frac{1}{p-1}\cdot\frac{p+1}{p}-\frac{1}{p}-1\\
        % &=-\frac{p+1}{p(p-1)}-\frac{p-1}{p(p-1)}-1\\
        % &=-\frac{2p}{p(p-1)}-1\\
        &=-\frac{2}{p-1}-1
        =-\frac{p+1}{p-1}<0.
    \end{align*}
    By continuity of $F(y)-y$ and $F'(y)-1$ on $J$, after shrinking $\delta_0$ if
    needed there exist constants $m,\eta>0$ such that
    \[
        F'(y)-1\le -2m
    \]
    holds for all $y\in J$, and
    \[
        F(C_p-\delta_0)-(C_p-\delta_0)\ge 2\eta,
        \qquad
        F(C_p+\delta_0)-(C_p+\delta_0)\le -2\eta,
    \]
    By shrinking $u_0$ further, we can ensure that
    \[
        u\sup_{y\in J}|G'(y)|+C_1u^\kappa\le m,\qquad
        \qquad
        u\sup_{y\in J}|G(y)|+C_0u^\kappa\le \eta,
        \qquad
        C_2u^{\kappa-1}\le 1
    \]
    for all $u\in(0,u_0]$.
    Then
    \[
        \frac{\partial}{\partial y} H(u,y) = F'(y)-1+uG'(y)+R_1(u,y) \le -2m + m = -m < 0
    \]
    holds for all $u\in(0,u_0]$ and $y\in J$.
    Furthermore,
    \begin{align*}
        H(u,C_p-\delta_0) &= F(C_p-\delta_0)-(C_p-\delta_0) + uG(C_p-\delta_0) + R_0(u,C_p-\delta_0) \\
        &\ge 2\eta - \eta = \eta > 0, \\
        H(u,C_p+\delta_0) &= F(C_p+\delta_0)-(C_p+\delta_0) + uG(C_p+\delta_0) + R_0(u,C_p+\delta_0) \\
        &\le -2\eta + \eta = -\eta < 0
    \end{align*}
    holds for all $u\in(0,u_0]$.
    Hence, by the intermediate value theorem and strict monotonicity in $y$, for
    each $u\in(0,u_0]$ there exists a unique $y^*(u)\in J$ such that
    $H(u,y^*(u))=0$, i.e.\ $\Phi(u,y^*(u))=y^*(u)$.
    Set $y^*(0)\defi C_p$.

    For $u>0$, the map $\Phi$ is $C^1$ in $(u,y)$ on $(0,u_0]\times J$, so the
    classical implicit function theorem applied at each point
    $(u,y^*(u))$ shows that $y^*$ is $C^1$ on $(0,u_0]$.

    In the next step we derive the expansion of $y^*(u)$ around $u=0$.
    First note that 
    \[
        \frac{G(C_p)}{1-F'(C_p)} = \frac{C_p/p+(p-1)/(2p^2)\,C_p^{q+1}}{1-F'(C_p)} = C_p\cdot \frac{(p-1)(3p+1)}{2p(p+1)^2} = D_p.
    \]
    Set $\tilde y(u)\defi C_p+D_pu$.
    Since $F$ and $G$ are $C^2$ on $J$, we can expand each term in
    \[
        H(u,\tilde y(u))
        =F(\tilde y(u))-\tilde y(u)+uG(\tilde y(u))+R_0(u,\tilde y(u)).
    \]
    Using $\tilde y(u)-C_p=D_pu$ and $F(C_p)=C_p$, Taylor's theorem gives
    \[
        F(\tilde y(u))-\tilde y(u)
        =\bigl(F'(C_p)-1\bigr)\,D_pu+O(u^2).
    \]
    Likewise,
    \[
        G(\tilde y(u))=G(C_p)+O(u),
        \qquad
        uG(\tilde y(u))=uG(C_p)+O(u^2),
    \]
    and from \eqref{eq:H-expand},
    \[
        R_0(u,\tilde y(u))=O(u^\kappa).
    \]
    Therefore
    \[
        H(u,\tilde y(u))
        = \bigl(F'(C_p)-1\bigr)D_p\,u + uG(C_p) + O(u^2) + O(u^\kappa).
    \]
    By the choice of $D_p$, the coefficient of $u$ vanishes, so
    \[
        H(u,\tilde y(u)) = O(u^\kappa)
    \]
    because $\kappa<2$.
    After shrinking $u_0$ once more, we have $\tilde y(u)\in J$ for all
    $u\in[0,u_0]$.
    By the mean value theorem, for each $u\in(0,u_0]$ there exists $\xi_u$ between
    $y^*(u)$ and $\tilde y(u)$ such that
    \[
        0 = H(u,y^*(u))
        = H(u,\tilde y(u))
        + \frac{\partial}{\partial y} H(u,\xi_u)\,\bigl(y^*(u)-\tilde y(u)\bigr).
    \]
    Since $\frac{\partial}{\partial y} H(u,\xi_u)\le -m$, it follows that
    \[
        |y^*(u)-\tilde y(u)| \le \frac{1}{m}\,|H(u,\tilde y(u))| = O(u^\kappa).
    \]
    Therefore
    \[
        y^*(u)=C_p+D_pu+O(u^\kappa),
    \]
    which in particular proves that $y^*$ is continuous at $u=0$.

    Finally, $z^*(u)=y^*(u)^q$ and a first-order expansion gives \eqref{eq:z-drift}.
    By \eqref{eq:Hu-expand} and \eqref{eq:R-bounds}, $|\frac{\partial}{\partial u} H(u,y)|\le M$ on $(0,u_0]\times J$ for some $M>0$.
    Implicit differentiation on $(0,u_0]$ yields
    \[
    (y^*)'(u) = -\frac{\partial_u H(u,y^*(u))}{\partial_y H(u,y^*(u))},
    \]
    hence $|(y^*)'(u)|\le M/m$ for all $u\in(0,u_0]$.
    Together with the continuity at $u=0$, this shows that $y^*$ is Lipschitz
    on $[0,u_0]$.
\end{proof}

\begin{lemma}\label{lem:slow-start-feasible}\linktoproof{lem:slow-start-feasible}
    Fix $p\ge 3$. There exists $u_{\mathrm{feas}}>0$ such that for all
    $u_0\in(0,u_{\mathrm{feas}}]$, the point $x_0^{\mathrm{slow}}(u_0)$ from \eqref{eq:slow-start-init} lies strictly inside $B_p$.
\end{lemma}
\begin{proof}\linkofproof{lem:slow-start-feasible}
    Write $x_0^{\mathrm{slow}}(u)=(1-u, w(u))$ with
    $w(u)=(C_p+D_pu)\,u^{1+\alpha}$ and $\alpha=(p-1)/p$.
    Define
    \[
        g(u)\defi \norm{x_0^{\mathrm{slow}}(u)}_p^p = (1-u)^p + |w(u)|^p, \qquad u\in[0,1].
    \]
    Moreover, since $p(1+\alpha)=2p-1>1$, the term $|w(u)|^p$ satisfies
    $|w(u)|^p=o(u)$ as $u\downarrow 0$ and thus $\left. \frac{\partial}{\partial u} |w(u)|^p \right|_{u=0} = 0$
    and
    \[
        g'(0)= \left. \frac{\partial}{\partial u} (1-u)^p \right|_{u=0} = -p<0.
    \]
    Since $g$ is continuous on $[0,1]$ and $g(0)=1$, there exists $u_{\mathrm{feas}}>0$ such that
    $g(u)<1$ for all $u\in(0,u_{\mathrm{feas}}]$, which implies $x_0^{\mathrm{slow}}(u)\in\operatorname{int}(B_p)$.
\end{proof}

\begin{lemma}\label{lem:phi-local-contraction}\linktoproof{lem:phi-local-contraction}
    Fix $p\ge 3$ and let $y^*$ be the fixed point curve from \Cref{prop:z-drift} with $y^*(u)=C_p+D_pu+O(u^\kappa)$. Then
    \begin{itemize}
        \item If $p>3$, then there exist $u_{\mathrm{ctr}}>0$, $\delta>0$, and
        $\lambda\in(0,1)$ such that for all $u\in(0,u_{\mathrm{ctr}}]$ and all $y$
        with $|y-y^*(u)|\le \delta$,
        \[
            \left|\frac{\partial}{\partial y}\Phi(u,y)\right|\le \lambda.
        \]
        \item If $p=3$, then for every $K_0>0$ there exist $u_{\mathrm{ctr}}>0$ and
        $c>0$ such that for all $u\in(0,u_{\mathrm{ctr}}]$ and all $y$ with
        $|y-y^*(u)|\le K_0 u^\kappa$,
        \[
            \left|\frac{\partial}{\partial y}\Phi(u,y)\right|\le 1-cu.
        \]
    \end{itemize}
\end{lemma}

\begin{proof}\linkofproof{lem:phi-local-contraction}
    By \Cref{lem:y-map}, $\frac{\partial}{\partial y}\Phi(u,y)=F'(y)+uG'(y)+O(u^\kappa)$ uniformly for
    $y$ in a compact neighborhood of $C_p$. By \Cref{prop:z-drift},
    $y^*(u)=C_p+D_pu+O(u^\kappa)$.

    If $p>3$, then $F'(C_p)=-2/(p-1)$ and $|F'(C_p)|<1$. By continuity of
    $\frac{\partial}{\partial y}\Phi$ in $(u,y)$, there exist $u_{\mathrm{ctr}}>0$, $\delta>0$, and
    $\lambda\in(0,1)$ such that $|\frac{\partial}{\partial y}\Phi(u,y)|\le\lambda$ whenever
    $u\in(0,u_{\mathrm{ctr}}]$ and $|y-y^*(u)|\le\delta$.

    Now assume $p=3$. Then $F'(C_3)=-1$.
    Using $y^*(u)=C_3+D_3u+O(u^\kappa)$ and expanding $F'(y^*(u))$ and $G'(y^*(u))$ around
    $C_3$ yields
    \begin{align*}
        F'(y^*(u)) &= F'(C_3) + F''(C_3)(y^*(u)-C_3) + \mathcal{O}((y^*(u)-C_3)^2) \\
        &= F'(C_3) + F''(C_3)D_3u + O(u^\kappa)
    \end{align*}
    and
    \begin{align*}
        uG'(y^*(u)) &= uG'(C_3) + uG''(C_3)(y^*(u)-C_3) + \mathcal{O}((y^*(u)-C_3)^2) \\
        &= uG'(C_3) + G''(C_3)D_3u^2 + O(u^\kappa)\\
        &= uG'(C_3) + O(u^\kappa).
    \end{align*}
    Therefore,
    \[
        \frac{\partial}{\partial y}\Phi(u,y^*(u)) = F'(C_3) + u\bigl(G'(C_3)+F''(C_3)D_3\bigr) + O(u^\kappa).
    \]
    For $p=3$ we have $q=3/2$ and $C_3^q=3/4$, so
    \[
        G'(C_3)
        = \frac13 + \frac{p-1}{2p^2}(q+1)C_3^q
        = \frac13 + \frac{5}{18}\cdot \frac34
        = \frac{13}{24},
        \qquad
        F''(C_3)D_3
        = C_3^{-1}\cdot \frac{5}{24}C_3
    = \frac{5}{24}.
    \]
    Thus $G'(C_3)+F''(C_3)D_3=18/24=3/4$. Therefore,
    \[
        \frac{\partial}{\partial y}\Phi(u,y^*(u)) = -1+\tfrac34 u + O(u^\kappa).
    \]
    Next, fix $K_0>0$ and consider $y$ with $|y-y^*(u)|\le K_0 u^\kappa$.
    By \Cref{lem:y-map}, $\frac{\partial}{\partial y}\Phi(u,y)=F'(y)+uG'(y)+O(u^\kappa)$ uniformly for
    $y$ in a compact neighborhood of $C_3$. Since $F'$ and $G'$ are $C^1$ there, we
    have $F'(y)=F'(y^*(u))+O(|y-y^*(u)|)$ and
    $G'(y)=G'(y^*(u))+O(|y-y^*(u)|)$, hence
    \[
        \frac{\partial}{\partial y}\Phi(u,y)
        = \frac{\partial}{\partial y}\Phi(u,y^*(u)) + O(u^\kappa) + O(K_0 u^\kappa)
        = -1+\tfrac34 u + O(u^\kappa).
    \]
    Since $\kappa=4/3>1$, we have $u^\kappa=o(u)$, so for all sufficiently small
    $u$ (depending on $K_0$), the function $\varepsilon(u,y)\defi \frac{\partial}{\partial y}\Phi(u,y)+1$
    satisfies $\varepsilon(u,y)\ge \tfrac38 u$.
    Therefore $\frac{\partial}{\partial y}\Phi(u,y)\in(-1,0)$ and
    $|\frac{\partial}{\partial y}\Phi(u,y)| = 1-\varepsilon(u,y) \le 1-\tfrac38 u$.
    This proves the claim with any $c\le \frac{3}{8}$.
\end{proof}

\begin{proposition}\label{prop:slow-start-tracking}\linktoproof{prop:slow-start-tracking}
    Fix $p\ge 3$.
    There exist $u_1>0$, a neighborhood $I\subset(0,p^{1/q})$ of $C_p$, and a constant
    $K>0$ such that for all $u_0\in(0,u_1]$, the FW trajectory from
    $x_0=x_0^{\mathrm{slow}}(u_0)$ satisfies for all $t\ge 0$:
    \[
        u_t\le u_0,\qquad y_t\in I,\qquad d_{1,t}>0,
        \qquad
        |y_t-y^*(u_t)| \le K u_t^\kappa.
    \]
    In particular,
    \begin{equation}
        \label{eq:z-drift-slow-start}
        z_t \defi y_t^q
        = \frac{p}{p+1} + \frac{p(3p+1)}{2(p+1)^3}\,u_t + o(u_t).
    \end{equation}
\end{proposition}

\begin{proof}\linkofproof{prop:slow-start-tracking}
    Fix a compact interval $I$ around $C_p$ such that $I\subset(0,p^{1/q})$.
    In the following we choose $u_1$ by iteratively shrinking to satisfy the desired properties.
    Since $y^*(0)=C_p$ and $y^*$ is continuous, we may assume $y^*([0,u_1])\subset \operatorname{int}(I)$ for sufficiently small $u_1$.
    Also shrink $u_1\le u_{\mathrm{feas}}$ from \Cref{lem:slow-start-feasible} so
    that $x_0^{\mathrm{slow}}(u_0)\in B_p$ for all $u_0\in(0,u_1]$.
    If $p>3$, let $(u_{\mathrm{ctr}},\delta,\lambda)$ be as in
    \Cref{lem:phi-local-contraction} and shrink $u_1\le u_{\mathrm{ctr}}$.

    We begin by showing that the \gls{fw} trajectory from $x_0^{\mathrm{slow}}(u_0)$ is well-defined
    and satisfies $d_{1,t}>0$.
    For $(u,y)\in(0,u_1]\times I$ we have $y^q<p$, hence by \eqref{eq:d1-expand}
    \[
        d_1 = u+(1+y^q u)^{-1/p}-1 = u\Bigl(1-\frac{y^q}{p}\Bigr)+O(u^2)>0
    \]
    for all sufficiently small $u_1$ (uniformly over $y\in I$).
    Since $\gamma\ge 0$, this implies $u' = u-\gamma d_1 \le u$ and thus
    $u_t\le u_0$ for all $t$.
    Moreover, since $w_0\neq 0$ and $d_{1,t}>0$, \Cref{cor:ratio-identity} implies
    $w_t\neq 0$ for all $t$, hence $y_t$ is well-defined for all $t$.

    Next we prove a uniform drift bound for $u$ showing that $u$ is decreasing slowly enough to ensure later that $y_t$ stays close enough to $y^*(u_t)$.
    By \eqref{eq:gamma-expand}, on the compact set $(0,u_1]\times I$ the unconstrained line-search minimizer satisfies $\gamma^*=O(u^\kappa)$ uniformly.
    Shrinking $u_1$ further ensures $\gamma^*\le 1$ on $(0,u_1]\times I$, so the
    exact line search over $\gamma\in[0,1]$ uses $\gamma=\gamma^*$ throughout
    this region and thus the identities in \Cref{prop:uw-dynamics,cor:ratio-identity} apply.
    Moreover, $d_1=O(u)$ uniformly on $(0,u_1]\times I$.
    Therefore there exists $B>0$ depending only on $p$ and $I$ such that
    \begin{equation}\label{eq:u-drift-bound}
        0 \le u-u' = \gamma d_1 \le B u^{\kappa+1}
    \end{equation}
    for all $(u,y)\in(0,u_1]\times I$.
    In particular, $u' \ge u(1-Bu^\kappa)$ and for $u_1$ small enough
    $u'\in(0,u_1]$ whenever $u\in(0,u_1]$.

    Now we can prove that $y_t$ stays close to $y^*(u_t)$, i.e. $|y_t-y^*(u_t)|\le K u_t^\kappa$ for some constant $K>0$.
    Let $e_t\defi y_t-y^*(u_t)$ and let $L>0$ be a Lipschitz constant of $y^*$
    on $[0,u_1]$.
    Since $\Phi$ is continuous, the intermediate value theorem implies there exists some $\xi_t$ between $y_t$ and
    $y^*(u_t)$ such that,
    \[
        y_{t+1}-y^*(u_t)
        = \Phi(u_t,y_t)-\Phi(u_t,y^*(u_t))
        = \frac{\partial}{\partial y}\Phi(u_t,\xi_t)\,e_t.
    \]
    Hence,
    \begin{align}
        \label{eq:e-recursion}
        |e_{t+1}|
        &= |y_{t+1}-y^*(u_{t+1})| \notag\\
        &\leq |y_{t+1}-y^*(u_t)| + |y^*(u_t)-y^*(u_{t+1})| \notag\\
        &\leq \left|\frac{\partial}{\partial y}\Phi(u_t,\xi_t)\right|\,|e_t| + L|u_{t+1}-u_t|.
    \end{align}
    We claim that for sufficiently small $u_0\le u_1$ there exists $K>0$ such that
    \begin{equation}\label{eq:tube}
        |e_t|\le K u_t^\kappa
        \qquad\text{for all }t\ge 0,
    \end{equation}
    which we prove by induction below.

    First note that $y_0=C_p+D_pu_0$ and $y^*(u)=C_p+D_pu+O(u^\kappa)$, so after shrinking $u_1$
    we may choose $K$ such that $|e_0|\le K u_0^\kappa$ for all $u_0\in(0,u_1]$.
    
    % If $p=3$, additionally invoke \Cref{lem:phi-local-contraction} with $K_0=K$ and
    % shrink $u_1$ so that the conclusion holds on $(0,u_1]$ with some constant $c>0$.

    Assume now inductively that $|e_t|\le K u_t^\kappa$.
    Then $\xi_t$ lies between $y_t$ and $y^*(u_t)$, so $|\,\xi_t-y^*(u_t)\,|\le K u_t^\kappa$.
    Next we apply \Cref{lem:phi-local-contraction} to bound $|\partial_y\Phi(u_t,\xi_t)|$.
    If $p>3$, shrink $u_1$ so that $K u_1^\kappa\le\delta$ and thus
    $|\,\xi_t-y^*(u_t)\,|\le K u_t^\kappa \leq \delta$.
    \Cref{lem:phi-local-contraction} then implies
    \[
        |\partial_y\Phi(u_t,\xi_t)|
        \le \lambda.
    \]
    If $p=3$, \Cref{lem:phi-local-contraction} with $K_0=K$ implies there exists $c>0$ such that
    \[
        |\partial_y\Phi(u_t,\xi_t)|\le 1-cu_t.
    \]
    Next we apply the uniform drift bound \eqref{eq:u-drift-bound} implying $|u_{t+1}-u_t|\le B u_t^{\kappa+1}$.
    Plugging these into \eqref{eq:e-recursion} gives
    \[
        |e_{t+1}|\le \lambda K u_t^\kappa + LB u_t^{\kappa+1}
    \]
    for $p>3$ and
    \[
        |e_{t+1}|\le (1-cu_t)K u_t^\kappa + LB u_t^{\kappa+1}
    \]
    for $p=3$.
    Using $u_{t+1}\ge u_t(1-Bu_t^\kappa)$, the inequality
    $(1-x)^\kappa\ge 1-\kappa x$ for $x\in[0,1]$ yields
    \[
        u_{t+1}^\kappa\ge u_t^\kappa(1-\kappa B u_t^\kappa).
    \]
    For $p>3$, choose $u_1$ small enough so that for all $u\in(0,u_1]$,
    $(1-\lambda)K \ge LB u + \kappa BK u^\kappa$ or equivalently
    $LBu + \lambda K \le K- \kappa BK u^\kappa$. This implies
    \begin{equation*}
        |e_{t+1}|
        \le \lambda K u_t^\kappa + LB u_t^{\kappa+1}
        = u_t^\kappa(K-\kappa BK u_t^\kappa)
        = K u_{t+1}^\kappa.
    \end{equation*}
    For $p=3$, choose $K$ large enough that $cK>LB$ and then choose $u_1$ small
    enough so that for all $u\in(0,u_1]$,
    $(cK-LB)u \ge \kappa BK u^\kappa$.
    This inequality again implies
    \begin{align*}
        |e_{t+1}|
        & \le (1-cu_t)K u_t^\kappa + LB u_t^{\kappa+1} \\
        & = u_t^\kappa (K - (cK - LB)u_t) \\
        & \le  u_t^\kappa (K - \kappa BK u_t^\kappa) \\
        & \le K u_{t+1}^\kappa. 
    \end{align*}
    In both cases, we obtain $|e_{t+1}|\le K u_{t+1}^\kappa$ completing the
    induction step and proving \eqref{eq:tube}.

    Next we shrink $u_1$ so that $K u_1^\kappa < \operatorname{dist}(y^*([0,u_1]), \partial I)$
    and thus
    \[
        |y_t-y^*(u_t)|\le K u_t^\kappa < \operatorname{dist}(y^*([0,u_1]), \partial I)
    \]
    for all $t\ge 0$.
    Since $u_t$ is decreasing and bounded below by $0$, we have $y^*(u_t)\in y^*([0,u_1])$ for all $t\ge 0$ and thus $y_t\in I$ for all $t\ge 0$.

    Finally, we prove that $u_t\downarrow 0$.
    Since $u_t$ is decreasing and bounded below by $0$, the limit
    $u_\infty\defi \lim_{t\to\infty}u_t$ exists.
    If $u_\infty>0$, then for all sufficiently large $t$ we have
    $u_t\in[u_\infty/2,u_1]$ and $y_t\in I$.
    On this compact set, the decrement $u-u'=\gamma d_1$ is a continuous and
    strictly positive function of $(u,y)$, hence has a positive minimum $\eta>0$.
    This would imply $u_{t+1}\le u_t-\eta$ for all large $t$ and thus $\{u_t\}_t$ would be unbounded, a contradiction.
    Therefore $u_t\downarrow 0$.

    Finally, \eqref{eq:tube} gives $e_t=O(u_t^\kappa)=o(u_t)$ (since $\kappa>1$),
    so $y_t=y^*(u_t)+o(u_t)$ and thus $z_t=y_t^q=z^*(u_t)+o(u_t)$.
    \Cref{eq:z-drift-slow-start} then follows from \eqref{eq:z-drift}.
\end{proof}

\begin{proposition}\label{prop:contraction-ap}\linktoproof{prop:contraction-ap}
    Along the slow-start trajectory of \Cref{prop:slow-start-tracking},
    \[
        1-s_t = a_p\,r_t^\kappa\,(1+o(1))
        \qquad\text{as }t\to\infty,
    \]
    where $a_p=2C_p^2=2\left(\frac{p}{p+1}\right)^{2(p-1)/p}$.
\end{proposition}

\begin{proof}\linkofproof{prop:contraction-ap}
    For a quadratic, exact line search gives the identity
    \[
        h_{t+1} = h_t - \frac{\ip{\nabla f(x_t),x_t-v_t}^2}{4\norm{v_t-x_t}^2}.
    \]
    In $(u,w)$ coordinates, $\nabla f(x_t)=2(-u_t,w_t)$ and $v_t-x_t=(d_{1,t},d_{2,t})$,
    so $\ip{\nabla f(x_t),x_t-v_t}=2(u_t d_{1,t}-w_t d_{2,t})$ and
    $\norm{v_t-x_t}^2=d_{1,t}^2+d_{2,t}^2$.
    By \Cref{prop:uw-dynamics}, $u_t d_{1,t}-w_t d_{2,t}=M_t-u_t+h_t$, where
    $M_t=(u_t^q+|w_t|^q)^{1/q}$.
    Therefore, writing $s_t^2= \frac{h_{t+1}}{h_t}$,
    \begin{equation}\label{eq:sq-identity}
        1-s_t^2
        = \frac{h_t - h_{t+1}}{h_t}
        = \frac{(M_t-u_t+h_t)^2}{h_t(d_{1,t}^2+d_{2,t}^2)}.
    \end{equation}
    First we consider \eqref{eq:M-u+h-expand} again and get the expansion
    \[
        M_t-u_t+h_t = u_t^2(1+\alpha z_t)+O(u_t^3)
    \]
    and hence
    \[
        (M_t-u_t+h_t)^2=u_t^4(1+\alpha z_t)^2(1+O(u_t)).
    \]
    Next we expand $h_t$.
    Along the slow-start trajectory, $y_t$ stays bounded and
    $w_t=y_t u_t^{1+\alpha}$, hence
    \begin{equation}
        \label{eq:h-expand}
        h_t=u_t^2 + w_t^2 = u_t^2(1+y_t^2u_t^{2\alpha})= u_t^2(1+O(u_t^\kappa)).
    \end{equation}
    Since $z_t \to \frac{p}{p+1}$ as $t\to\infty$ and is therefore bounded, we have
    \[
        |v_{2,t}|=\rho_t^{q-1}(1+\rho_t^q)^{-1/p}=(z_t u_t)^{1/p}(1+z_t u_t)^{-1/p} = \Theta(u_t^{1/p})
    \]
    and
    \[
        |w_t|=y_t u_t^{1+\alpha}=O(u_t^{1+\alpha}).
    \]
    By \eqref{eq:d2-abs-expand} and using $1+\alpha - \frac{1}{p} = \kappa$, we have
    \[
        |d_{2,t}|=|v_{2,t}|+|w_t| = |v_{2,t}| \left(1+\frac{|w_t|}{|v_{2,t}|}\right) = |v_{2,t}|(1+O(u_t^\kappa)),
    \]
    and thus
    \[
        d_{2,t}^2 = (z_t u_t)^{2/p}(1+z_t u_t)^{-2/p}\,(1+O(u_t^\kappa)) = \Theta(u_t^{2/p}).
    \]
    By \eqref{eq:d1-expand}, we have $d_{1,t}=O(u_t)$ which together with $2-\frac{2}{p}=\kappa$ yields
    \[
        d_{1,t}^2 + d_{2,t}^2 = d_{2,t}^2 \left(1 + \frac{d_{1,t}^2}{d_{2,t}^2} \right)
        = d_{2,t}^2 \left( 1 + \frac{O(u_t^2)}{\Theta(u_t^{2/p})} \right) = d_{2,t}^2(1+O(u_t^\kappa)).
    \]
    Combining with $h_t=u_t^2(1+O(u_t^\kappa))$ and \eqref{eq:sq-identity} yields
    \begin{align*}
        1-s_t^2
        &= \frac{(M_t-u_t+h_t)^2}{h_t(d_{1,t}^2+d_{2,t}^2)} \\
        &= \frac{u_t^4(1+\alpha z_t)^2(1+O(u_t))}
        {u_t^2\,(z_tu_t)^{2/p}(1+z_tu_t)^{-2/p}(1+O(u_t^\kappa))} \\
        &= \frac{(1+\alpha z_t)^2}{z_t^{2/p}}\,u_t^\kappa\,
        (1+z_tu_t)^{2/p}\,\frac{1+O(u_t)}{1+O(u_t^\kappa)} \\
        &= \frac{(1+\alpha z_t)^2}{z_t^{2/p}}\,u_t^\kappa\,(1+O(u_t)),
    \end{align*}
    where we have used $4-(2+2/p)=2-2/p=\kappa$ in the second last step.
    Equivalently,
    \[
        1-s_t^2 = 2A_0(z_t)\,u_t^\kappa\,(1+O(u_t)),
        \qquad
        A_0(z)\defi \frac{(1+\alpha z)^2}{2z^{2/p}}.
    \]
    Since the right hand side converges to zero as $t\to\infty$, we have $s_t \to 1$ and thus $1+s_t=2+o(1)$.
    Hence,
    \[
        1-s_t = \frac{1-s_t^2}{1+s_t}
        = A_0(z_t)\,u_t^\kappa\,(1+o(1)).
    \]
    Using $z_t=\tfrac{p}{p+1}+o(1)$, we obtain $A_0(z_t)=A_0(\tfrac{p}{p+1})+o(1)$.
    A direct evaluation gives $A_0(\tfrac{p}{p+1})=2(p/(p+1))^{2(p-1)/p}=2C_p^2=a_p$.
    Finally, $r_t=u_t(1+O(u_t^\kappa))$ implies $u_t^\kappa=r_t^\kappa(1+o(1))$, completing the proof.
\end{proof}

\lowerBoundTheorem*
\begin{proof}\linkofproof{thm:lower-bound}
    Fix $p\ge 3$, set $\mathcal C\defi B_p$, and let $f$ be the model in \eqref{eq:model}. For any $x^*\in\arg\min_{x\in\mathcal C }f(x)$, define $h_t\defi f(x_t)-f(x^*)$ along the exact-line-search \gls{fw} trajectory from $x_0=x_0^{\mathrm{slow}}(u_0)$.
    Define $\psi_t\defi h_t^{-\alpha}$.
    Since $\psi_{t+1}=s_t^{-2\alpha}\psi_t$ we have
    \[
    \Delta\psi_t\defi \psi_{t+1}-\psi_t=\psi_t(s_t^{-2\alpha}-1).
    \]
    From \Cref{prop:contraction-ap}, we have $1-s_t=a_pr_t^\kappa(1+o(1))$.
    Since $s_t\to1$, a first-order expansion at $s=1$ gives
    \[
      s_t^{-2\alpha}-1 = 2\alpha(1-s_t)+O((1-s_t)^2).
    \]
    Therefore
    \begin{align*}
      \Delta\psi_t
      &= \psi_t\bigl(s_t^{-2\alpha}-1\bigr) \\
      &= \psi_t\Bigl(2\alpha(1-s_t)+O((1-s_t)^2)\Bigr) \\
      &= 2\alpha a_p\,\psi_t r_t^\kappa(1+o(1))
      + O\!\left(\psi_t r_t^{2\kappa}\right).
    \end{align*}
    Using $r_t^\kappa=h_t^\alpha=\psi_t^{-1}$, we obtain
    \[
      \psi_t r_t^\kappa=1,
      \qquad
      \psi_t r_t^{2\kappa}=r_t^\kappa\to0,
    \]
    and hence
    \[
      \Delta\psi_t = 2\alpha a_p + o(1)=\kappa a_p+o(1).
    \]
    Summing yields $\psi_T\sim \kappa a_p T$, hence
    \[
        h_T=\psi_T^{-1/\alpha}\sim (\kappa a_p T)^{-1/\alpha} = (\kappa a_p T)^{-p/(p-1)}.
    \]
\end{proof}

\begin{proposition}\label{prop:naive-power-extension}\linktoproof{prop:naive-power-extension}
    Let $x_0\in B_p$ and let $(x_t)$ be the exact-line-search \gls{fw} trajectory for $f(x)=\|x-e_1\|_2^2$ started from $x_0$. For any $\theta\in(0,\frac12]$ and $\mu>0$, running exact-line-search \gls{fw} on
    \[
        g(x)=\mu^{-1/\theta}\|x-e_1\|_2^{1/\theta}
        =\mu^{-1/\theta}f(x)^{1/(2\theta)}
    \]
    from the same initialization $x_0$ produces exactly the same iterates $(x_t)$.
\end{proposition}

\begin{proof}\linkofproof{prop:naive-power-extension}
    Let $\theta\in(0,\tfrac12]$ and $\mu>0$, and define $g(x)=\mu^{-1/\theta}f(x)^{1/(2\theta)}$.
    Let $x\neq x^*$. Since $\nabla f(x)=2(x-e_1)$ and $f(x)=\|x-e_1\|_2^2$,
    \[
        \nabla g(x)=\frac{1}{2\theta}\mu^{-1/\theta}f(x)^{1/(2\theta)-1}\nabla f(x).
    \]
    The scalar prefactor is strictly positive, hence
    \[
        \argmin_{v\in B_p}\ip{\nabla g(x),v}
        =
        \argmin_{v\in B_p}\ip{\nabla f(x),v}.
    \]
    So both objectives produce the same \gls{lmo} atom at every non-optimal iterate.
    At $x^*$, both gradients vanish and the method is stationary for both objectives.

    Fix $t$ and denote the common atom by $v_t$. For $\gamma\in[0,1]$,
    \[
        g\!\left(x_t+\gamma(v_t-x_t)\right)
        =
        \mu^{-1/\theta}
        f\!\left(x_t+\gamma(v_t-x_t)\right)^{1/(2\theta)}.
    \]
    Since $f\ge 0$ and $s\mapsto s^{1/(2\theta)}$ is increasing on $\mathbb R_+$,
    \[
        \argmin_{\gamma\in[0,1]} g\!\left(x_t+\gamma(v_t-x_t)\right)
        =
        \argmin_{\gamma\in[0,1]} f\!\left(x_t+\gamma(v_t-x_t)\right).
    \]
    Therefore exact line search yields the same step size $\gamma_t$ for $f$ and $g$.
    By induction on $t$, both runs generate exactly the same sequence $(x_t)$ from the same initialization.
\end{proof}

\extendedLowerBound*
\begin{proof}\linkofproof{thm:extended-lower-bound}
    Fix $p\ge 3$, $\mu>0$ and $\theta\in(0,\tfrac12]$, and set
    \[
        \C\defi B_p,\qquad
        g(x)\defi \mu^{-1/\theta}\|x-e_1\|_2^{1/\theta}.
    \]
    Let $x^*=e_1$. Since $\theta\in(0,\tfrac12]$, $g$ is smooth on $\C$.

    We first verify the $(\mu,\theta)$-\gls{heb}. For any $x\in\C$,
    \[
        g(x)-g(x^*)
        =\mu^{-1/\theta}\|x-e_1\|_2^{1/\theta},
    \]
    hence
    \[
        \mu\bigl(g(x)-g(x^*)\bigr)^\theta
        =\mu\Bigl(\mu^{-1/\theta}\|x-e_1\|_2^{1/\theta}\Bigr)^\theta
        =\|x-e_1\|_2
        =\|x-x^*\|_2.
    \]
    Therefore $\|x-x^*\|_2\le \mu\bigl(g(x)-g(x^*)\bigr)^\theta$ for all $x\in\C$.

    Now consider exact-line-search \gls{fw} initialized as in \Cref{thm:lower-bound}. By \Cref{prop:naive-power-extension}, the trajectory for $g$ coincides with the one for $f(x)=\|x-e_1\|_2^2$. Thus the iterate sequence for $g$ is exactly $(x_t)$ from \Cref{thm:lower-bound}.

    Using $g(x^*)=0$ and \Cref{thm:lower-bound},
    \begin{align*}
        g(x_T)-g(x^*)
        &=\mu^{-1/\theta}\bigl(f(x_T)-f(x^*)\bigr)^{1/(2\theta)}\\
        &\sim
        \mu^{-1/\theta}
        \left(\left(\frac{p+1}{p}\right)^2
        \left(\frac{p}{4(p-1)}\right)^{\frac{p}{p-1}}
        T^{-\frac{p}{p-1}}\right)^{\frac{1}{2\theta}}\\
        &=
        \mu^{-1/\theta}\left(\frac{p+1}{p}\right)^{\frac{1}{\theta}}
        \left(\frac{p}{4(p-1)}\right)^{\frac{p}{2\theta(p-1)}}
        T^{-\frac{p}{2\theta(p-1)}}.
    \end{align*}
    This proves the asymptotic in \eqref{eq:extended-lower-bound}. The $\Omega$ statement follows immediately.
\end{proof}

\end{document}